\documentclass[a4paper,10pt]{article}
\usepackage[utf8]{inputenc}
\usepackage{amsthm}
\usepackage{amsmath}
\usepackage{authblk}
\usepackage{mathtools}
\usepackage[mathscr]{euscript}
\usepackage{cite}
\usepackage{amssymb}
\usepackage{nicefrac}
\usepackage[pdftex]{graphicx}
   \usepackage[justification=centering]{caption}
   \usepackage{float}
\usepackage{array}
\usepackage[ruled,vlined]{algorithm2e}
\newtheorem{theorem}{Theorem}

\newtheorem{lemma}{Lemma}

\newcommand{\pro}{\text{{\LARGE $\sqcap$}}}

\newtheorem{remark}{Remark}
\title{Optimization over time-varying networks with unbounded delays
}
\author{Arunselvan Ramaswamy% <-this % stops a space
\thanks{This work was supported by the German Research Foundation (DFG)
(project number 315248657)}% <-this % stops a space
\thanks{Department of Electrical Engineering and Computer Science, Paderborn University, Paderborn - 33102, Germany
        {\tt\small arunr@mail.upb.de}} \\
        Adrian Redder \thanks{Department of Electrical Engineering and Computer Science, Paderborn University, Paderborn - 33102, Germany
        {\tt\small aredder@mail.upb.de}} \\
        Danie E. Quevedo \thanks{Department of Electrical Engineering and Computer Science, Paderborn University, Paderborn - 33102, Germany
        {\tt\small dquevedo@ieee.org}}
}
\begin{document}
\maketitle
\begin{abstract}
Solving optimization problems in multi-agent systems (MAS) involves information exchange between agents. These solutions must be robust to delays and errors that arise from an unreliable wireless network which typically connects the MAS. In today's large-scale dynamic Internet Of Things style multi-agent scenarios, the network topology changes and evolves over time. In this paper, we present a simple distributed gradient based optimization framework and an associated algorithm. Convergence to a minimum of a given objective is shown under mild conditions on the network topology and objective. Specifically, we only assume that a message sent by a sender reaches the intended receiver, possibly delayed, with some positive probability. To the best of our knowledge ours is the first analysis under such weak general network conditions. We also discuss in detail the verifiability of all assumptions involved. This paper makes a significant technical contribution in terms of the allowed class of objective functions. Specifically, we present an analysis wherein the objective function is such that its sample-gradient is merely locally Lipschitz continuous. The theory developed herein is supported by numerical results. Another contribution of this paper is a consensus algorithm based on the main framework/ algorithm. The long-term behavior of this consensus algorithm is a direct consequence of the theory presented. Again, we believe that ours is the first consensus algorithm to account for unbounded stochastic communication delays, in addition to time-varying networks.
\end{abstract}
%%%%%%%%%%%%%%%%%%%%%%%%%%%%%%%%%%%%%%%%%%%%%%%%%
%%%%%%%%%%%%%%%%%%%%%%%%%%%%%%%%%%%%%%%%%%%%%%%%%
%%%%%%%%%%%%%%%%%%%%%%%%%%%%%%%%%%%%%%%%%%%%%%%%%
\section{INTRODUCTION}\label{sec_introduction}
Systems consisting of multiple autonomous agents that interact with each other to solve local and global problems are called Multi-agent Systems (MAS). The agents involved may operate on different spatial and temporal scales, although in a cooperative manner. Examples include, smart electricity grids and buildings, vehicular and mobile networks, and the Internet of Things (IoT). Typically, MAS are large-scale in nature and spread over large geographical areas. It is therefore convenient and cost-effective to facilitate agent interactions through the use of wireless communication networks. Although cost effective and easy to set-up, wireless networks are prone to delays and errors. 

The agents interact and cooperate with each other to solve a system-level problem, while performing local computations and taking local decisions. System-level problems such as multi-agent learning, distributed networked control, consensus, etc., may be solved by viewing them as distributed optimization problems. Here, the classical optimization paradigm is not applicable. This is because it assumes the presence of a central entity that has access to all available problem data. In the scenario considered herein, a central entity may or may not exist. When present, it does not have access to all problem data. It may however request and obtain data from other agents in the system. As stated earlier, all communications are over a lossy delayed wireless network, whose topology and quality of service (QoS) may be time-varying. The reader is referred to \cite{nedic2018distributed} for a survey of distributed optimization algorithms for networked control problems.

In this paper, we are interested in solving the following optimization problem in a decentralized manner:
\begin{equation}\label{intro_opt_prob} 
x^* = (x_1^*, \ldots, x_D^*) = \underset{x \in \mathbb{R}^d}{\text{argmin }} F(x),
\end{equation}
where $D$ is the number of agents in the MAS, $F$ is a given stochastic global objective function,  $x^*_i$ is the local optimal decision variable of agent-i for $1 \le i \le D$, and $d$ is the sum of the system dimensions of the $D$ agents. Starting from an initial random estimate $x_i ^0$, of $x_i ^*$,  agent-i iteratively refines its estimates such that it converges to $x_i ^*$, together with other agents simultaneously. Put together, these local optima minimize the global objective. 

Broadly speaking, the distributed delayed gradient-based optimization procedure presented here entails the following. At time $n$, agent-i uses local estimates $x_j ^{m(j)}$, $j \neq i$, $0 \le m(j) \le n$, obtained from other agents via the wireless network. Ideally, agent-i would like to have $x_j ^n$, $j \neq i$. However, due to network delays and errors, it may only have access to older estimates, hence we have the condition that $0 \le m(j) \le n$. Due to the large system size, every pair of agents may \emph{not} be connected by a dedicated wireless channel. Instead, information is relayed through other agents in the system, leading to further delays. Put succinctly, we show that the optimization problem can be solved through distributed delayed gradient descent, provided there is a positive probability of successful information exchange between every pair of agents. This probability may vary over time, thereby allowing for time-varying network topologies.

The literature on distributed optimization is too long to list. We only mention a few recent ones and compare them to the results found herein. An accelerated algorithm to tackle distributed optimization problems wherein local functions themselves are a sum, is studied in \cite{hendrikx2019asynchronous}. In \cite{wang2019distributed}, a Kiefer--Wolfowitz style distributed descent algorithm is developed to solve non-convex and non-smooth optimization problems in a distributed manner. Although \cite{wang2019distributed} considers time-varying network topologies, the assumptions on network quality and topology are strict. Specifically, they assume that (a) the associated network graph is doubly stochastic at every time-step, (b) the probability of successful transmissions across any link is lower bounded by a fixed constant that does not vary over time, and (c) every pair of agents must be periodically connected by a direct link/ channel. On the other hand, in this paper, we merely assume that there is a positive probability associated with inter-agent information exchange, that may vary over time. While we do assume that the objective function is differentiable, we may drop this assumption and use Kiefer-Wolfowitz style finite difference terms, similar to \cite{wang2019distributed},  or SPSA-C terms similar to \cite{ramaswamy2018analysis}, instead. The reader is referred to Remark~\ref{asmp_r2} for a short discussion on this. In another related work, \cite{yin2018switching}, the network topology switching process is assumed to have a unique stationary distribution. Although we do not explicitly consider network switching, our assumptions imply that there is no requirement for the aforementioned unique stationary distribution. There may be multiple limiting distributions or the switching process may even oscillate.

Another problem considered in this paper is that of consensus in MAS. Here, the agents seek to find a common control objective starting from a random one, through interactions and local computations. The cumulative consensus problem is given by:
\[
x^* = \underset{x \in \mathbb{R}^d}{\text{argmin }} \sum \limits_{i=1}^D f_i(x),
\]
where $f_i$ is a local objective of agent-i that is not accessible to others. This problem is studied in \cite{sirb2018decentralized} under the added complexity that only delayed local gradients are available. Further, these delays may be stochastic and unbounded. Subgradient based consensus algorithms are studied in \cite{nedic2009distributed}, constraints and time-delays are tackled in \cite{hou2018constrained}, and rate of convergence is studied in \cite{nedic2010convergence}. The theory developed herein, to tackle distributed optimization, is utilized to develop a simple distributed gradient-based consensus algorithm. It accommodates unbounded stochastic delays, even with respect to information exchanged between agents. This latter type of delay is not considered in \cite{sirb2018decentralized}. Again, we show that consensus can be achieved in the presence of time-varying network topologies.

\subsection{Our Contributions}
We have developed simple sufficient conditions for convergence of distributed (approximate) gradient-based algorithms, used to solve large-scale distributed stochastic optimization problems. For this, we have used tools from Stochastic Approximation Algorithms \cite{benaim2005stochastic, borkar2009stochastic, borkar1998asynchronous} and Viability Theory \cite{aubin2012differential}. The theory accounts for stochastic unbounded delays in the information exchanged between agents, that usually arises from the underlying wireless network. The analysis can be trivially extended to accommodate for computational delays in obtaining local gradients. This theory naturally yields an algorithm that uses distributed delayed gradients. In Remark~\ref{asmp_r2}, we discuss how this algorithm can be extended to use Kiefer-Wolfowitz style gradient approximations instead of exact gradients. 

Convergence is shown under very mild conditions on communication. Specifically, we assume that there is a positive probability of successful, possibly delayed, information exchange between agents. We also assume that the transmissions across different channels are independent. Unlike assumptions that are popular in literature, these conditions naturally account for a time-varying network topology. They also accommodate multi-hop networks that may connect agents. In cases where the topology change is governed by a switching process, our algorithm is guaranteed to converge even when this process has multiple stationarities. Finally, the theory is used to develop an easy to implement consensus algorithm.

In the next section, we set up the problem of interest and state it formally. We then present the general form of an iteration, that uses delayed gradient information, to solve a given stochastic distributed optimization problem. This iteration will later form the basis for the development of the previously mentioned algorithms.
%%%%%%%%%%%%%%%%%%%%%%%%%%%%%%%%%%%%%%%%%%%%%%%%%%%%
%%%%%%%%%%%%%%%%%%%%%%%%%%%%%%%%%%%%%%%%%%%%%%%%%%%%
%%%%%%%%%%%%%%%%%%%%%%%%%%%%%%%%%%%%%%%%%%%%%%%%%%%%
\section{Distributed optimization over time-varying networks} \label{sec_usc}
We consider a $D-$agent system denoted by $\mathcal{V} \coloneqq \{1, \ldots, D\}$. The collective goal of this D-agent system is to minimize a given global stochastic objective function $F$, in a cooperative manner. As stated before, the minimizer $x^*$ of $F$, is composed of $x^* _i$, $1 \le i \le D$, where $x^* _i$ is interpreted as a local optimal decision variable with respect to agent-i. The decision space of agent-i is given by $\mathbb{R}^{d_i}$, such that $\sum \limits_{i=1}^D d_i = d$. Each agent starts with a completely random estimate $x^0_i \in \mathbb{R}^{d_i}$, of $x^*_i$, and refines this iteratively through gradient-based searches in its decision space. Specifically, the descent direction is given by the partial derivative taken with respect to the $x_i$-component, i.e., $\nabla_{x_i}F$. To calculate this partial derivative, at every stage, it requires estimate information from other agents. This information is obtained using a wireless network that is prone to losses and delays. In other words, at time $n$, agent-i refines its estimate using the following gradient descent step:
%We consider a $D-$agent system denoted by $\mathcal{V} \coloneqq \{1, \ldots, D\}$. The decision variable $x_i$, of agent-i, belongs to $\mathbb{R}^{d_i}$ and $\sum \limits_{i=1}^D d_i = d$. The general consensus problem requires that the agents find optimal decision variables, in a cooperative manner, to minimize some global stochastic cost function, say $F$. Iterative solutions to the consensus problem typically involve generating local estimates (of the decision variables), using relevant information from other agents.
%Gradient-based iterative procedures generate these local estimates by performing local gradient descent steps. A naive generic consensus algorithm based on gradient descent, requires agent-i to perform the following:
\begin{equation}
\label{usc_unreal}
x^{n+1}_i = x^n _i - a(n) \nabla_{x_i} F(x^{m(1)} _1, \ldots, x^{m(D)} _D), \text{ where}
\end{equation}
\textbf{(i)} $F:\mathbb{R}^d \to \mathbb{R}$ is the given stochastic cost function,\\
\textbf{(ii)} $x^{m(j)}_j$ is the $m(j)^{th}$ estimate of agent-j, $1 \le j \le D$, with $0 \le m(j) \le n$. \\
\textbf{(iii)} $\nabla_{x_i} (\cdotp)$ denotes the partial derivative with respect to vector $x_i$,\\
\textbf{(iv)} $\{a(n)\}_{n \ge 0}$ is the given step-size sequence.

In \eqref{usc_unreal} the delay $m(j)$ is anywhere between $0$ and $n$ ($0 \le m(j) \le n$) to account for information delays. Due to the large system size, each agent-pair may not be connected by a dedicated wireless channel. Instead, information is exchanged by forwarding it through other agents along a path connecting them, which causes further delays. 

\textbf{Graph view of MAS. }\emph{A MAS that is connected by a wireless network can be represented by a graph.} There is one node for every agent in the MAS. Two nodes are connected by an edge \textit{if and only if}  the corresponding agents can communicate directly. Communication between a pair of agents is considered ``direct'' when there is no involvement from other agents. In real-world scenarios this may be realized when a pair is connected by a dedicated wireless channel, or a multi-hop network that does not contain other agents. Further, the channels may be unidirectional which are represented using directional edges. Hence, in the most general case, a MAS is represented by a digraph (directed graph) \cite{west1996introduction}. The channels themselves, in general, may be modeled as hidden Markov fading channels \cite{turin1998hidden}. It may be noted that a wide range of communication models are accommodated by the framework presented herein. Further, we allow for time-varying network topologies and QoS, to account for real-world scenarios wherein physical wear-and-tear, introduction of new resources, etc., happen routinely.

\emph{In this paper, we are interested in the development of an analytical framework, and an associated algorithm, to solve the distributed optimization problem in MAS with unreliable communication, time-varying network topology and QoS. For this, among others, we develop simple verifiable practical assumptions on the wireless network, that ensure convergence of distributed delayed gradient based optimization algorithms.}

\textbf{Stochastic Global Objective.} In \eqref{usc_unreal}, the form of the stochastic objective function $F$ is given by: $$F(x) \coloneqq \mathbb{E}_{\xi} [ f(x, \xi) ],$$ where $f: \mathbb{R}^d \times \mathcal{S} \to \mathbb{R}$ is a real-valued function and  $\xi$ is an $\mathcal{S}$ valued random variable. Typically, $\xi$ is the randomness that arises due to environmental and network-related fluctuations. We shall broadly call $\xi$ as the \emph{network noise}. In this paper, we assume that all agents may only observe noisy samples of the stochastic objective, at any point in time. Note that stochasticity in the objective function arises due to the random variable $\xi$. Further, when two agents observe the function at the same time, the sample obtained is influenced by the same realization of $\xi$. In words, say that agent-i wishes to evaluate $\nabla_{x_i}F$ at $\tilde{x}_i$. At the same time agent-j, $j \neq i$, wishes to evaluate $\nabla_{x_j}F$ at $\tilde{x}_j$. Further, say that the sample realization of the network noise is $\xi_0$. Then, the agent-i partial-derivative is $\nabla_{x_i}F(\tilde{x}_i, \xi_0)$, while the agent-j one is $\nabla_{x_j}F(\tilde{x}_j, \xi_0)$. It must however be noted that the analysis presented herein can be readily extended to accommodate local fluctuations that may vary with the observing agent. The reader is referred to Remark~\ref{local_fluc} for a brief discussion on this.

\textbf{Autonomy and asynchronicity.} An important trait of MAS is the autonomy with which the agents involved operate. Agent-i performs the above partial derivative calculations at every \emph{time-step} to search for a local optimal in its decision space. For this, it requires estimates from other agents that may be delayed or lost on account of the unreliable communication network. Another important aspect of a MAS is that the agents are asynchronous, i.e., the agents may not synchronize their clocks. The implication of this is that the \emph{time-step}, previously mentioned, is with respect to the local clock of agent-i. However, convergence of any optimization algorithm operating in this setting requires some causal assumptions on the relative update frequency of the agents. For clarity, we carry out our analysis with respect to a \emph{hypothetical global clock}, which is faster than all agent clocks and the network clock. In particular, for every tick of the hypothetical clock the agent/network clocks tick at most once. Although delay was captured in \eqref{usc_unreal}, it does not account for asynchronicity. We therefore rewrite \eqref{usc_unreal} to accommodate for asynchronicity and to change notations for convenience.

%If one were to consider real-world scenarios,  Further, the underlying communication topology may change over time. The stochastic nature of the underlying communication network is reflected through the use of a stochastic cost function. In this paper, we are interested in stochastic cost functions such as   The generic consensus algorithm does not account for these stochastic quantities. Another limitation of \eqref{usc_unreal} is the requirement that the latest estimates of every agent be available to every other agent in the system. If the agents communicate via a wireless network, this requirement may be hard to ensure due to the presence of information losses and delays. Finally,\eqref{usc_unreal} requires that all agents synchronize their clocks, a hard requirement to ensure in large-scale multi-agent systems. 

%Below, we present an asynchronous modification of \eqref{usc_unreal} to overcome its limitations. This modified consensus algorithm allows for (a) asynchronicity between agents (b) wide ranging communication architectures and (c) stochastic cost functions.
\begin{multline}
\label{usc_real}
x^{n+1}_i = x^n _i - a(\nu(n, i)) I(i \in Y^n) \\ \left( \nabla_{x_i} f(x^{n - \tau_{1i}(n)} _1, \ldots, x^{n - \tau_{Di}(n)} _D, \xi^n)  + \epsilon^n _i \right),
\end{multline}
where \\
\textbf{(i)} $\xi^n$'s are i.i.d. $\mathcal{S}$-valued samples, due to the stochastic communication network. \\
\textbf{(ii)} $Y^n \subset \{1, 2, \ldots, D\}$ represents the subset of agents that update their estimates between times $n-1$ and $n$. Note that $n$ is the time index with respect to the hypothetical global clock and $Y^n$ accounts for asynchronicity by tracking agent updates. Say that the agent-i clock did not tick between $n-1$ and $n$ (no update of the local estimate), then $i \notin Y^n$, $I(i \in Y^n)=0$ and $x^{n+1}_i = x^n _i$.\\
\textbf{(iii)} $0 \le \tau_{ji}(n) \le n$ is the (stochastic) delay experienced by agent-i in receiving the agent-j estimate, at time $n$.\\
\textbf{(iv)} $\nu(n,i) = \sum \limits_{m=0}^{n} I(i \in Y^n)$ is the number of times that agent-i updated its local estimate until time $n$. For example, if agent-2 has updated its estimate $3$ times when the global clock has ticked $10$ times, then $\nu(10,2) = 3$.\\
\textbf{(v)} $\epsilon^n_i$ is a stochastic additive error term that arises due to sampling and/or gradient estimations. 

The reader may recognize \eqref{usc_real} as a distributed asynchronous sample-based noisy gradient descent scheme. The global clock indexed by $n$ is hypothetical and only used for analysis. In the following section, we present a quick convergence analysis of \eqref{usc_real}. Specifically, we present conditions on $\tau_{ij}(n)$, for $1 \le i,j \le D$ and $n \ge 0$, such that \eqref{usc_real} converges a minimum of $F$. Clearly, the delay random variables $\tau_{ij}(n)$ depend on the characteristics of the underlying communication network. Additionally, it depends on the ``communication protocol'' used (who communicates and when). 

Till now, we have not focused on the connecting wireless network, specifically on the graph associated with the MAS. In Section~\ref{sec_algo}, we use the theory developed in Section~\ref{sec_gen_analysis} to present sufficient conditions on the MAS-graph/wireless network, to ensure convergence of \eqref{usc_real}. \emph{This is one of the major contributions of this paper.} Hitherto requirements, in literature, on the MAS-graph are too strong and often impractical, see Section~\ref{sec_introduction}. We essentially show convergence under the mild assumption that, at any time, the wireless network must connect any two agents with positive probability, possibly with delays.

\textbf{Communication Protocol (CP).} In this paper, the set of rules governing the exchange of information, in a MAS, using the underlying wireless network is called the communication protocol. For example, the CP presented here is the following. At every tick of the network clock, agent-i forwards a list containing the latest available estimates from all agents to its neighbors in the MAS-graph. Put succinctly, it forwards $\hat{x}_{ji},\  1\le j \le D$, where $\hat{x}_{ji}$ is the agent-j estimate available with agent-i, that is possibly old, see Section~\ref{sec_algo} for details. In addition to network quality and topology, communication delays are also affected by the CP used. In general, the CP must be topology sensitive, since leveraging topological knowledge can greatly reduce delays and hasten the rate of convergence, of distributed optimization algorithms. In this paper, however, we present a topology agnostic CP, merely for the sake of simplicity and generality. In Section~\ref{sec_algo}, we only present assumptions on the network and all requirements on the CP used, are intrinsic.

In the following section, we present the sufficient conditions for convergence, analyze the behavior of \eqref{usc_real}, under them. In Section~\ref{sec_algo} we present a distributed stochastic optimization algorithm and an associated communication protocol. In Section~\ref{sec_empirical} we present experimental results for the algorithm presented. Finally, in Section~\ref{sec_cons} we present a consensus algorithm motivated by the theory developed in Sections~\ref{sec_gen_analysis}, complete with relevant empirical evidence.
%\textit{The general consensus algorithm given by \eqref{usc_real} updates the local estimate of each agent using the last known estimates of other agents. The local estimates are communicated over a wireless network, governed by some communication protocol.} Typically, a consensus algorithm needs to specify the communication protocol which dictates how and when information is exchanged between various agents. We have indirectly specified a protocol in \eqref{usc_real}, by incorporating delay random variables $\tau_{ij}(n)$ for $n \ge 0$ and $1 \le i, j \le D$. In the following section, we present assumptions on $\tau_{ij}(n)$. In addition, assumptions on the relative update frequency of agents, sampling errors, etc. are presented, under which \eqref{usc_real} is analyzed in Section~\ref{sec_ana}. 

%It may be noted that the theory developed in Sections~\ref{sec_asmp} and \ref{sec_ana} can be generally applied to analyze gradient-based iterative learning algorithms. \textit{They are also used to develop a simple and easy-to-implement consensus algorithm, complete with a communication protocol in Section~\ref{sec_algo}}. In the following section, we present assumptions required for the analysis of the asynchronous consensus algorithm.

\section{Analysis of \eqref{usc_real}} \label{sec_gen_analysis}
First, we rewrite \eqref{usc_real} in the following form to make it more conducive for analysis:
\begin{multline}
\label{asmp_real_alt}
x^{n+1}_i = x^n _i - a(\nu(n, i)) I(i \in Y^n) \\ ( \mathbb{E} \left[ \nabla_{x_i} f(x^{n - \tau_{1i}(n)} _1, \ldots, x^{n - \tau_{Di}(n)} _D, \xi^n) \mid \mathcal{F}^n \right] + \\ \epsilon^n _i + M_i ^{n+1} ),
\end{multline}
%\begin{flushleft}
where \\
\textbf{(a)} $\mathcal{F}^0 \coloneqq  \sigma \langle x^0, \epsilon^0, Y^0 \rangle$ and $ \mathcal{F}^n \coloneqq \sigma \langle x^m, \epsilon^m, Y^m, \xi^{k} \mid \ m \le n, \ k < n \rangle$ for $n \ge 1$.
 \begin{flushleft} 
 \textbf{(b)} $M_i ^{n+1} \coloneqq \nabla_{x_i} f(x^{n - \tau_{1i}(n)} _1, \ldots, x^{n - \tau_{Di}(n)} _D, \xi^n)$   $-\mathbb{E} \left[ \nabla_{x_i} f(x^{n - \tau_{1i}(n)} _1, \ldots, x^{n - \tau_{Di}(n)} _D, \xi^n) \mid \mathcal{F}^n \right]$. 
\end{flushleft}

\vspace*{.5cm}

To see that \eqref{asmp_real_alt} and \eqref{usc_real} are equivalent, we observe that $\mathbb{E} \left[ \nabla_{x_i} f(x^{n - \tau_{1i}(n)} _1, \ldots, x^{n - \tau_{Di}(n)} _D, \xi^n) \mid \mathcal{F}^n \right] = \nabla_{x_i}F(x^{n - \tau_{1i}(n)} _1, \ldots, x^{n - \tau_{Di}(n)} _D, \xi^n)$ for $n \ge 0$ and $1 \le i \le D$. The aforementioned equality follows trivially from the definition of the filtration, $\{\mathcal{F}^n\}_{n \ge 0}$. As stated earlier, $\{ \xi^n \}_{n \ge 0}$ are $\mathcal{S}$-valued i.i.d. (independent and identically distributed) samples. They represent network and environment fluctuations. In Remark~\ref{local_fluc}, we briefly discuss the case where different agents are affected by different local network fluctuations, i.e., agent-i is affected by $\xi_i ^n$ at time $n$.
%Note that  is the sample of a random variable associated with the environment / communication network, at time $n$. As stated earlier, we assume that all agents observe the same sample at any given time. The theory presented herein can be easily extended to account for agent-level stochasticity.
\subsection{Assumptions on \eqref{usc_real}} \label{sec_asmp}
\begin{itemize}
 \item[{\bf (A1)}] 
 \begin{itemize}
 \item[(i)] {\it $\nabla_x f$ is continuous.}
 \item[(ii)] {\it It is locally Lipschitz continuous in the $x$-coordinate and \emph{may change} with the second coordinate.}
 \end{itemize}
 \item[{\bf (A2)}] {\it The step-size sequence $\{a(n)\}_{n \ge 0}$ satisfies the following conditions:}
 \begin{itemize}
  \item[(i)] {\it $\sum \limits_{n \ge 0} a(n) = \infty$ and $\sum \limits_{n \ge 0} a(n) ^2 < \infty$.}
  \item[(ii)] {\it $\limsup \limits_{n \to \infty} \sup \limits_{y \in [x, 1]} \frac{a( \lfloor yn \rfloor)}{a(n)} < \infty$
  for $0 < x \le 1$.}
  \item[(iii)] $\sup \limits_{n \ge 0} a(n) \le 1$.
 \item[(iv)] For $m \le n$, we have $a(n) \le \kappa a(m)$, where $\kappa > 0$.
 \item[(v)] There exists $1/2 < \eta < 1$ and a non-negative integer-valued random variable $\overline{\tau}$ such that:
	    \begin{itemize}
	     \item[(a)] $a(n) = o(n ^{- \eta})$.
	     \item[(b)] $\overline{\tau}$ stochastically dominates all $\tau_{kl}(n)$ and satisfies
	     \[
	      E\left[ \overline{\tau}^{1/\eta} \right] < \infty.
	     \]
	      \end{itemize}
 \end{itemize}
 \item[{\bf (A3)}] {\it $\sup \limits_{n \ge 0} \ \lVert x^n \rVert < \infty$ a.s.}
 \item[{\bf (A4)}] {\it $\liminf \limits_{n \to \infty} \frac{\nu(n, i)}{n} > 0$ for $1 \le i \le D$.}
  \item[{\bf (A5)}] {\it Almost surely, $\limsup \limits_{n \to \infty} \ \lVert \epsilon^n \rVert \le \epsilon$ for some fixed $\epsilon > 0$.}
\end{itemize}

\subsection{Remarks on the assumptions}
\begin{remark}
Let us say that we are given a distributed delayed gradient-based iterative optimization algorithm. We first check whether: \textbf{(i)} the local updates of this algorithm are described by \eqref{usc_real} \textbf{(ii)} the information delays due to the associated communication protocol and network delays/losses satisfy (A2)(v). If these two conditions are satisfied, then the analyses associated with \eqref{usc_real} can be applied to understand its behavior. The other assumptions are typically satisfied in many real-world scenarios. In Section~\ref{sec_algo}, we refine (A2)(v) to obtain a practical verifiable condition (A6). This refinement encompasses the requirements on the network topology and QoS, and CP.
\end{remark}
\begin{remark}
\label{asmp_r2}
Assumption (A5) requires that the stochastic additive error terms be norm-bounded asymptotically. These error terms may arise due to sampling or due to the use of approximate gradients instead of exact ones. Let us suppose that some or all $\nabla_{x_i} f \left(x^{n - \tau_{1i}(n)} _1, \ldots, x^{n - \tau_{Di}(n)} _D, \xi^n \right)$, $1 \le i \le D$, are unavailable or cannot be easily computed. Further, let us suppose that each agent may sample the objective function. Then a SPSA-C style \cite{ramaswamy2018analysis} gradient approximator may be used in its stead. This is given by $$\frac{f \left( \cdotp+ \Delta, \xi^n \right) - f \left(\cdotp - \Delta, \xi^n \right)}{2c \Delta_i},$$ where $0 < c < \infty$ is a perturbation parameter, and $\Delta$ is a d-dimensional random perturbation vector with  $i^{th}$ component $\Delta_i$. Note that since $\tau_{ii}(n)$ is assumed to be zero without loss of generality, $x^n _i = x^{n - \tau_{ii}(n)} _i$. The approximation errors associated with using this gradient estimator are shown to be in $o(c^2)$, see \cite{ramaswamy2018analysis} for details. Further, these errors are captured by the $\epsilon$ sequence, and satisfy (A5). In particular, the analysis and the algorithms presented can be readily extended to approximate gradient methods.
\end{remark} 
\begin{remark}
\label{local_fluc}
In \eqref{usc_real}, each agent is affected by the same realization of network fluctuations. In reality, however, network fluctuations may be local and vary from agent to agent. Hence, $\xi^n$ in \eqref{usc_real} must be replaced by $\xi^n _i$ (specific to agent-i). Let us assume that $\{\xi^n_i\}_{n \ge 0, 1\le i \le D}$ are independent. Then, we may define $\mathcal{F}^0 \coloneqq \sigma \langle x^0, \xi^0, Y^0 \rangle$ and $\mathcal{F}^n \coloneqq \sigma \left\langle x^m, \xi^m, Y^m, \xi^k_i \mid m \le n, k<n, 1 \le i \le D \right\rangle$, $n \ge 1$. Using this newly defined filtration, we can show that
\begin{multline*} \nabla_{x_i} F(x^{n - \tau_{1i}(n)} _1, \ldots, x^{n - \tau_{Di}(n)} _D, \xi^n) =  \\ \mathbb{E} \left[ \nabla_{x_i} f(x^{n - \tau_{1i}(n)} _1, \ldots, x^{n - \tau_{Di}(n)} _D, \xi^n) \mid \mathcal{F}^n \right]. \end{multline*}
The analysis that follows will now hold with minor modifications.
\end{remark}
\subsection{Preliminaries and lemmata}
The assumptions (A1)--(A5) are motivated by the theory developed by \cite{ramaswamy2018asynchronous}. To apply the theory from  \cite{ramaswamy2018asynchronous}, for the analysis of \eqref{usc_real}, we present a couple of technical lemmata. Lemma~\ref{asmp_L} is needed, since unlike \cite{ramaswamy2018asynchronous}, we allow for objectives such that $\nabla_x f$ is locally Lipschitz continuous in the $x$-coordinate. This is because local Lipschitz continuity is a property that is easy to verify. Further, the constant associated with every $\hat{x} \in \mathbb{R}^d$ may vary with $\xi$. Through this lemma, we guarantee the existence of a global Lipschitz constant, that does not vary with $x$ or $\xi$, under the following restrictions: $x$ belongs to a compact convex subset $\mathcal{K} \subset \mathbb{R}^d$ (restrict the decision search space to a sample path dependent compact convex subset), and $\mathcal{S}$ is such that some one-point-compactification theorem is applicable. Note that we state the theorem for compact $\mathcal{S}$  and discuss how its proof can be extended to general $\mathcal{S}$ using one-point-compactification theorems. Lemma~\ref{asmp_noise_conv} shows that the errors arising from the use of gradient samples, as opposed to expected gradients, vanish asymptotically.

To apply the theory from \cite{ramaswamy2018asynchronous}, we rewrite \eqref{asmp_real_alt} in the following form: 

\begin{multline}
\label{asmp_cons}
x^{n+1} = x^n -  \overline{a}(n) \lambda^n \\ \left(\nabla_{x} F(x^{n - \tau_{1i}(n)} _1, \ldots, x^{n - \tau_{Di}(n)} _D, \xi^n)  + \epsilon^n + M ^{n+1} \right),
\end{multline}
\\ where $\overline{a}(n) := \underset{i \in Y^n}{\max} \ a(\nu(n, i))$, 
$q(n, i) := \frac{a(\nu(n, i)) I(i \in Y^n)}{\overline{a}(n)}$ and
\[
\lambda^n := 
\left[ \begin{array}{rrr}
\mbox{\LARGE $Q_1^n$} & \dots & \mbox{\LARGE $0$} \\
\vdots & \ddots & \vdots \\
\mbox{\LARGE $0$} & \dots & \mbox{\LARGE $Q_D^n$}
\end{array} \right]_{d \times d}, \text{ with}
\]
\[
Q^n_i := \left[ \begin{array}{rrr}
q(n,i) & \dots & 0 \\
\vdots & \ddots & \vdots \\
0 & \dots & q(n,i)
\end{array} \right]_{d_i \times d_i} .
\]
Recall that there are $D$ agents in the system, the dimension of agent-i is $d_i$ and $\sum \limits_{i = 1}^D d_i = d$. For $1 \le i \le d$ and $n \ge 0$, $Q^n_i$ is a diagonal matrix of dimension $d_i \times d_i$, such that every diagonal element equals $q(n, i)$. For $n \ge 0$, $\lambda^n$ is a $d \times d$ diagonal matrix obtained by placing the $Q_i^n$ matrices along its diagonal. The term $q(n,i)$ can be interpreted as the step-size of agent-i at time $n$, in relation to the maximum step-size of an active agent $\overline{a}(n)$. The $\lambda^n$ matrix is used to capture the transient relative update of various agents with respect to $\overline{a}(n)$. The various agent updates are related to each other via this maximum step-size. To analyze the discrete-time algorithm \eqref{usc_real}, we use $\{\overline{a}(n)\}_{n \ge 0}$ to divide the time axis, and obtain a corresponding trajectory in the ``continuous domain''. Here, the ``limit'' of $\lambda^n$ captures the asymptotic relative update frequency of the agents.
%%%%%%%%%%%%%%%%%%%%%
%%%%%%%%%%%%%%%%%%%%%
%%%%%%%%%%%%%%%%%%%%%
\begin{lemma}
\label{asmp_L}
Given a compact convex set $\mathcal{K} \subset \mathbb{R}^d$, there exists $0< L< \infty $ such that $\lVert \nabla_x f(y, \xi) - \nabla_x f(z, \xi) \rVert \le L \lVert x - y \rVert$ $\forall \ x,y \in \mathcal{K}$ and $\xi \in \mathcal{S}$. Further, if $\mathcal{S}$ is compact, then $L$ is independent of $\xi$, although may depend on $\mathcal{K}$.
\end{lemma}
\begin{proof}
From $(A1)(ii)$ we get for every $z \in \mathcal{K}$ and $\xi \in \mathcal{S}$, there exists $\infty> r(z) > 0$ and $\infty > L(z, \xi) > 0$ such that $\lVert \nabla_x f(y, \xi) - \nabla_x f(z, \xi) \rVert \le L(z, \xi) \lVert y - z \rVert$ for every $y \in B_{ r(z)}(z)$ and $\xi \in \mathcal{S}$. Without loss of generality, $L(z, \xi)$ is such that $\lVert \nabla_x f(y^1, \xi) - \nabla_x f(y^2, \xi) \rVert \le L(z, \xi) \lVert y^1 - y^2 \rVert$ for every $y^1, y^2 \in B_{ r(z)}(z)$. Since $\mathcal{K}$ is compact, the open cover $\{B_{r(z)/2}(z) \mid z \in \mathcal{K}\}$ has a finite sub-cover given by $\mathcal{S}_c \coloneqq \{B_{r(z^j)/2}(z^j) \mid 1 \le j \le M, \ z^j \in \mathcal{K}\}$. Further, define $\hat{L}:= \max \limits_{1 \le j \le M} L(z^j, \xi)$.

To prove the statement of the lemma, we consider some $y^1, y^2 \in \mathcal{K}$. It follows from the convexity assumption that $y^1$ and $y^2$ are connected by a straight line $\ell \subset \mathcal{K}$. It follows from the discussion above, that we can find a sequence of sets, $B_{r(z^{j_1})}(z^{j_1}), \ldots, B_{r(z^{j_m})}(z^{j_m})$ from $\mathcal{S}_c$, covering the aforementioned line $\ell$, such that $m \le M$, $y^1 \in B_{r(z^{j_1})}(z^{j_1})$, $y^2 \in B_{r(z^{j_m})}(z^{j_m})$, and $B_{r(z^{j_i})}(z^{j_i}) \cap B_{r(z^{j_{i+1}})}(z^{j_{i+1}}) \neq \emptyset$ for $1 \le i \le m-1$. Now, we construct a sequence $z^0, \ldots, z^{m}$ such that $z^0 = y^1$, $z^{m} = y^2$ and $z^i \in B_{r(z^{j_i})}(z^{j_i}) \cap B_{r(z^{j_{i+1}})}(z^{j_{i+1}}) \cap \ell$, $1 \le i \le m-1$. Using this sequence, we get the following:
\[
\lVert \nabla_x f(y^1, \xi) - \nabla_x f(y^2, \xi) \rVert \le \sum \limits_{i = 0}^{m-1} \lVert \nabla_x f(z^{i+1}, \xi) - \nabla_x f(z^{i}, \xi) \rVert,
\]
\[
\lVert \nabla_x f(y^1, \xi) - \nabla_x f(y^2, \xi) \rVert \le \sum \limits_{i = 0}^{m-1} \hat{L} \lVert z^{i+1} - z^{i} \rVert \text{and}
\]
\[
\lVert \nabla_x f(y^1, \xi) - \nabla_x f(y^2, \xi) \rVert \le m \hat{L} \lVert y^1 - y^2 \rVert.
\]
If we define $\tilde{L} \coloneqq M \hat{L}$, then it may depend on both $\mathcal{K}$ and $\xi$. 

From the above discussions, we can find a minimum $L(\xi)$ such that 
\[
\lVert \nabla_x f(y^1, \xi) - \nabla_x f(y^2, \xi) \rVert \le L(\xi) \lVert y^1 - y^2 \rVert
\]
holds for every $y^1, y^2 \in \mathcal{K}$ with $L(\xi) < \infty$ for all $\xi \in \mathcal{S}$. Note that $L(\xi)$ depends on $\mathcal{K}$, but this dependency is omitted in the notation for the sake of clarity in presentation.

It is now left to show that there exists a constant that depends only on $\mathcal{K}$ and not on $\xi$. To do this, we show that the map $\xi \mapsto L(\xi)$ is continuous. Consequently, we may choose $\sup \limits_{\xi \in \mathcal{S}} L(\xi)$ as the $\xi$-independent Lipschitz constant.
First, we need to show that it satisfies: $\liminf \limits_{\xi^n \to \xi} L(\xi^n) = L(\xi)$. To see this, observe that:
\begin{equation}\label{lemma1_eq1}
\lVert \nabla_x f(y^1, \xi) - \nabla_x f(y^2, \xi) \rVert \le \lVert \nabla_x f(y^1, \xi^n) - \nabla_x f(y^2, \xi^n) \rVert + \epsilon^n,
\end{equation}
where $\epsilon^n = \lVert \nabla_x f(y^1, \xi) - \nabla_x f(y^1, \xi^n) \rVert + \lVert \nabla_x f(y^2, \xi) - \nabla_x f(y^2, \xi^n) \rVert$. Clearly $\lim \limits_{\xi^n \to \xi} \epsilon^n = 0$, additionally recall that $\lVert \nabla_x f(y^1, \xi^n) - \nabla_x f(y^2, \xi^n) \rVert \le L(\xi^n) \lVert y^1 - y^2 \rVert$. Now, taking $\liminf$ on both sides of \eqref{lemma1_eq1}, we get
\[
\lVert \nabla_x f(y^1, \xi) - \nabla_x f(y^2, \xi) \rVert \le \liminf \limits_{\xi^n \to \xi} L(\xi^n) \lVert y^1 - y^2 \rVert.
\]
It follows from the minimality of $L(\xi)$ that $L(\xi) \le \liminf \limits_{\xi^n \to \xi} L(\xi^n)$. It is left to show that $L(\xi) \ge \liminf \limits_{\xi^n \to \xi} L(\xi^n)$. Let us assume the contrary. Then, there exists $\delta_0 > 0$ such that $L(\xi) + \delta_0 < \liminf \limits_{\xi^n \to \xi} L(\xi^n)$. Hence there exists $N$ such that $\forall \ n \ge N$ $L(\xi) + \delta_0 < L(\xi^n)$. It follows from the minimality of $L(\xi^n)$ that there exists $\tilde{y}, \tilde{z} \in \mathcal{K}$ such that $\lVert \nabla_x f(\tilde{y}, \xi^n) - \nabla_x f(\tilde{z}, \xi^n) \rVert > (L(\xi^n) - \delta_0/2) \lVert \tilde{y} - \tilde{z}\rVert $. Taking $\liminf \limits_{\xi^n \to \xi}$ on both sides, we get $\lVert \nabla_x f(\tilde{y}, \xi) - \nabla_x f(\tilde{z}, \xi) \rVert > (L(\xi) - \delta_0/2) \lVert \tilde{y} - \tilde{z}\rVert $. This contradicts the minimality of $L(\xi)$. Hence, $\liminf \limits_{\xi^n \to \xi} L(\xi^n) = L(\xi)$. 

We make a further claim that $\lim \limits_{\xi^n \to \xi} L(\xi^n)$ exists and equals $L(\xi)$. For this, we show that every sub-sequence of $\{L(\xi^n)\}_{n \ge 0}$, such that $\xi^n \to \xi$, has a further subsequence whose limit is $L(\xi)$. Given a subsequence $\{L(\xi^{m(n)})\}_{n \ge 0}$, $\{m(n)\}_{n \ge 0} \subset \{n \ge 0\}$, it follows from the above discussion that $\liminf \limits_{\xi^{m(n)} \to \xi} L(\xi^{m(n)}) = L(\xi)$. Hence it has a further subsequence that converges to $L(\xi)$. In other words, the map $\xi \mapsto L(\xi)$ is continuous.

Now define $L := \max \limits_{\xi \in \mathcal{S}} L(\xi)$. It follows from the continuity of the $\xi \mapsto L(\xi)$ map, that $L < \infty$. The statement of the lemma is now immediate.
\end{proof}
%%%%%%%%%%%%%%%%%%%%%
%%%%%%%%%%%%%%%%%%%%%
%%%%%%%%%%%%%%%%%%%%%
The above lemma states that the local Lipschitz property of $\nabla_x f$ is global provided we restrict the decision space to compact convex subsets of $\mathbb{R}^d$. For clarity in presentation, we assume that $\mathcal{S}$ is compact. In addition to being realistic, we do not lose any generality by making this assumption. This is because one can use one-point compactification theorems when $\mathcal{S} = \mathbb{R}^n$. 

Define $\zeta^n := \sum \limits_{m=0}^{n-1} \overline{a}(m) \lambda^m M^{m+1}$. The sequence $\{\zeta^n\}_{n \ge 0}$ must be considered due to the use of samples instead of expected values. We will show that $\{(\zeta^n, \mathcal{F}^n )\}_{n \ge 0}$ is a zero-mean square integrable Martingale difference sequence, and use the Martingale convergence theorem to show that it converges almost surely. This is important, since this implies that the asymptotic behavior of \eqref{asmp_cons} is identical to the version that uses expected values.
\begin{lemma} \label{asmp_noise_conv}
Almost surely $\lim \limits_{n \to \infty} \zeta^n$ exists, where $\zeta^n = \sum \limits_{m=0}^{n-1} \overline{a}(m) \lambda^m M^{m+1}$ for $n \ge 0$.
\end{lemma}
\begin{proof}
To prove the above statement, we first show that
$\{(\zeta^n, \mathcal{F}^n )\}_{n \ge 0}$ is a zero-mean square integrable Martingale difference sequence. Next, we show that its associated quadratic variation process is bounded almost surely. The convergence of $\zeta^n$ is then a direct consequence of the Martingale convergence theorem, \cite{durrett2019probability}. Before proceeding, we make a few important observations. First, we note that it follows from $(A3)$, that $x^n \in B_R(0)$ for some sample path dependent $R < \infty$. Using Lemma~\ref{asmp_L} we can associate $L$ with $B_R(0)$. Hence, for $z \in B_R(0)$, we get:
\[
\lVert \nabla_x f(z, \xi) - \nabla_x f(0, \xi) \rVert \le L \lVert z \rVert, \text{ and}
\]
\[
\lVert \nabla_x f(z, \xi) \rVert \le \rVert \nabla_x f(0, \xi) \rVert + L \lVert z \rVert.
\]
Since $\mathcal{S}$ is compact, we get that 
\begin{equation}
\label{asmp_K}
\lVert \nabla_x f(z, \xi) \rVert \le K (1+  \lVert z \rVert), 
\end{equation}
where $K < \infty$ is sample path dependent. For our second observation, we note that $\sum \limits_{n \ge 0} \overline{a}(n)^2 < \infty$ as a consequence of $(A2)(i) $ and $(A4)$. Finally, we observe that the diagonal entries in $\lambda^n$ are all less than $1$. 

It follows from the definition of $M_{n+1}$ that $\mathbb{E} \zeta^n = 0$. It follows from the above made observations and \eqref{asmp_K} that $\mathbb{E} \lVert \zeta^n \rVert^2 < \infty$ a.s. Hence, we have that $\zeta^n$ is a zero-mean square integrable Martingale difference sequence with filtration $\mathcal{F}_n$. To show that the associated quadratic variation process is bounded, it is enough to show that 
\[
\sum \limits_{n \ge 0} \lVert \zeta^{n+1} - \zeta^n \rVert ^2 < \infty \text{ a.s.}
\]
But this is a direct consequence of the above made observations. Hence, it follows from the Martingale Convergence Theorem that $\lim \limits_{n \to \infty} \zeta^n$ exists a.s.
\end{proof}

An important component of our analysis is the verifiability of our assumptions. There are three of them whose demonstrability is unclear: $(A1)(ii)$, $(A2)(v)$ and $(A3)$. We clarify the verifiability of $(A1)(ii)$ in the following section. We defer the discussion on the verifiability of $(A2)(v)$ till the end Section~\ref{sec_algo}. This is because, we require the exact assumption on the network topology and quality, for this discussion. Finally, note that we do not discuss the verifiability of $(A3)$. The readers may however refer to Section 5 of \cite{ramaswamy2018asynchronous} for verifiable sufficient conditions for $(A3)$.

\subsection{On the verifiability of $(A1)(ii)$} \label{sec_verify}
Assumption (A1)(ii) requires that $\nabla_x f$ be locally Lipschitz continuous in the $x$-coordinate. This is a gross weakening of similar ones found in literature that require global Lipschitz property, see for e.g., \cite{borkar2006stochastic}. Since twice continuous differentiability is a sufficient condition for the local Lipschitz property, (A1)(ii) is easily verifiable. Additionally, in literature, the global Lipschitz constant is assumed to be independent of the second coordinate ($\xi$ coordinate). However, (A2)(ii) allows for the local Lipschitz constant to vary with $\xi$. 

On the surface, it seems surprising that (A1)(ii) is able to allow for a two-fold weakening, without compromising verifiability. Typical assumptions in literature do not exhibit these two traits simultaneously. For (A1)(ii), this is not surprising in the light of Lemma~\ref{asmp_L}. In particular, its proof leverages two ideas. First, it uses the stability assumption (A3) to move from local to global Lipschitzness, by restricting the set of interest, to a sample path dependent compact convex set obtained from (A3). Second, it eliminates the dependency on $\xi$ by assuming, without loss of generality, that $\mathcal{S}$ is compact. This is done because $\mathcal{S}$ is typically $\mathbb{R}^n$, $n \ge 1$, or more generally a locally compact space. In this case, one uses Alexandroff's one point compactification theorem \cite{bredon2013topology} to find a homeomorphic map into a subspace $\mathcal{S}_{*}$, of a compact space. Also, $f$ is easily extended to be viewed as a map from $\mathbb{R}^d \times \mathcal{S}_{*}$. Then, instead of viewing samples as arising in $\mathcal{S}$, one uses corresponding samples in $\mathcal{S}_{*}$. Now that $\mathcal{S}_{*}$ is compact, the statement of Lemma~\ref{asmp_L} holds.

\subsection{Convergence Analysis of \eqref{usc_real}} \label{sec_ana}
We are finally ready to analyze \eqref{usc_real}. As stated earlier, our assumptions and lemmata are developed and presented so that the theory from \cite{ramaswamy2018asynchronous} can be used in the analysis. The reader must note that the notations used in this section are consistent those in \cite{ramaswamy2018asynchronous}. To avoid redundancies, we do not repeat overlapping proofs and refer the reader to corresponding ones from \cite{ramaswamy2018asynchronous}.

%In this section, we present a convergence analysis of \eqref{usc_real} under the assumptions listed in Section~\ref{sec_asmp}. The iteration given by \eqref{usc_real} describes a general form of stochastic gradient based distributed asynchronous optimization algorithm. Its structure takes into account the existence of an underlying communication graph, albeit intrinsically. In particular, (A2)(v) states that all communication delays must be stochastically dominated by a second-moment-bounded random variable. However, this assumption cannot be verified unless assumptions on the structure and topology of the underlying communication network and on the specified communication protocol, are clearly stated. All these specifications together with an optimization routine, governed by \eqref{usc_real},  constitute the distributed optimization algorithm. This, we present in Section~\ref{sec_algo}.

%The general asynchronous consensus algorithm given by \eqref{usc_real} does not specify any communication protocol, directly. In Section~\ref{sec_algo} we present a consensus algorithm and an associated communication protocol, based on \eqref{usc_real}. Here, in this section, we present an analysis of \eqref{usc_real} under the assumptions listed in the previous section. First, we study the effect of communication delays, captured by the $\tau_{ij}(n)$ random variables. Note that unlike previous literature on stochastic consensus, we allow for unbounded delays. 

At time $n$, agent-i performs gradient descent in its local decision space using $(x_1 ^{n - \tau_{1i}(n)}, \ldots, x_D ^{n - \tau_{Di}(n)})$, where $x_j ^{n - \tau_{ji}(n)}$ is the latest agent-j estimate that is available with agent-i at time $n$. The error incurred at time $n$ with respect to agent-i, $e^n_i$, is given by:
\begin{multline*}
e^n_i \coloneqq a(\nu(n, i)) I(i \in Y^n) \times \\ \lVert \nabla_x f(x_1 ^n, \ldots, x_D ^n, \xi^n) -  f(x_1 ^{n - \tau_{1i}(n)}, \ldots, x_D ^{n - \tau_{Di}(n)}, \xi^n)\rVert.
\end{multline*}
It follows from Lemma~\ref{asmp_L} that $\nabla_x f$ is globally Lipschitz continuous when restricted to $\overline{B}_R(0)$, where $x^n \in \overline{B}_R(0)$ for $n \ge 0$, and $0 < R < \infty$ is obtained from $(A3)$. Hence,
\begin{multline} \label{ana_eq1}
e^n_i \le a(\nu(n, i)) L \lVert x^n - (x_1 ^{n - \tau_{1i}(n)}, \ldots, x_D ^{n - \tau_{Di}(n)}) \rVert.
\end{multline}
To bound the right hand side of the above equation, it is enough to consider $\lVert x_j ^n - x_j ^{n - \tau_{ji}(n)} \rVert$, $1 \le j \le D$, separately. Using assumptions $(A3)$, $(A2)(iv)$ we get that almost surely
\begin{multline*}
\lVert x_j ^n - x_j ^{n - \tau_{ji}(n)} \rVert \le \sum \limits_{m = n - \tau_{ji}(n)}^{n-1} \lVert x_j ^{m+1} - x_j ^m \rVert \le C \tau_{ji}(n),
\end{multline*}
for some sample path dependent $C < \infty$. To proceed with the analysis, we need the following technical lemma.
\begin{lemma} \label{ana_BC}
$P(\tau_{ji}(n) > n^\eta \ i.o.) = 0$.
\end{lemma}
\begin{proof}
To prove the above statement, we show that $\sum \limits_{n \ge 0} P(\tau_{ji}(n) > n^\eta) < \infty$ and invoke the Borel-Cantelli lemma \cite{durrett2019probability}.

Recall that $(A2)(v)$ requires the existence of $\eta$ such that $\mathbb{E} \tau^{1/ \eta} < \infty$ and $a(n) \in o(n^{-\eta})$. It is easy to see one can choose $\eta$ such that $\mathbb{E} \tau^{1/ (\eta - \delta)} < \infty$, $1/2 \le \eta - \delta$ and $(A2)(v)$ hold, using a very small $\delta$. Let us define $\beta \coloneqq \eta - \delta$, then $\mathbb{E} \tau^{1/ \beta} < \infty$. Also, recall that $\tau_{ij}(n)$ is stochastically dominated by $\tau$ for $n \ge 0$ and $1 \le i, j \le D$. Hence,
\[
P(\tau_{ji}(n) > n^\eta) \le P(\tau > n^\eta).
\]
It follows from the Markov inequality that
\[
P(\tau > n^\eta) \le \frac{\mathbb{E} \tau ^{1/\beta}}{n^{\eta / \beta}}, \text{ and that}
\]
\[
\sum \limits_{n \ge 0} P(\tau_{ji}(n) > n^\eta) \le \mathbb{E} \tau ^{1/\beta} \sum \limits_{n \ge 0} \frac{1}{n^{\eta / \beta}}.
\]
We have that $\{ \nicefrac{1}{n^{\eta / \beta}} \}_{n \ge 1}$ is a summable sequence, since $\beta < \eta$ by construction. Hence we get that $\sum \limits_{n \ge 0} P(\tau_{ji}(n) > n^\eta) < \infty$, as required.
\end{proof}
Informally, the above lemma states that there is a sample path dependent $0 < N < \infty$ such that $\tau_{ij}(n) \le n^\eta$ for all $n \ge N$. From (A2)(v)(a) and (A4), we get that almost surely $a(\nu(n, i)) \tau_{ij}(n) \in o(1)$ for $n \ge N$. If we combine these observations with \eqref{ana_eq1}, we get that $e^n \in o(1)$ for $n \ge N$. In other words, the errors associated with delays vanish asymptotically.

\textit{The remainder of the analysis, in this section, overlaps with the analysis presented in Section 4 of \cite{ramaswamy2018asynchronous}. Hence, we merely outline the important steps involved and refer the reader to \cite{ramaswamy2018asynchronous} for details}.

Using (A1)--(A5) and hitherto presented discussions, we get that \eqref{usc_real} converges to a closed connected invariant set associated with
\[
\dot{x}(t) \in \Lambda(t) \nabla_x F(x(t)) + \overline{B}_\epsilon (0),
\]
where $\epsilon$ is the asymptotic norm-bound on the additive error terms. It may be noted that the $\Lambda$ process may be thought of as capturing the relative update frequency of the agents involved, in the asymptotic sense.

It follows from (A4) that all the agents are updated in the same order of magnitude. Hence, for $1 \le i,j \le D$, one can ensure that:
\[
\lim \limits_{n \to \infty} \frac{\sum \limits_{m=0}^n a(\nu(m, i))}{\sum \limits_{m=0}^n a(\nu(m, j))} \text{ \ \ \ \ \ exists.}
\]
When the above equation is satisfied, the step-size sequence is said to be ``balanced''. An important role of (A4) is to ensure that the given steps-size sequence is ``balanced''.  Then, it follows from the theory developed in \cite{borkar1998asynchronous} that
\[
\Lambda(t) = \begin{pmatrix} 
\nicefrac{1}{D} & \dots & 0 \\
\vdots & \ddots & \vdots \\
0 & \dots & \nicefrac{1}{D}
\end{pmatrix}, \ t \ge 0.
\]
The diagonal matrix in the above equation is denoted by $diag(1/D, \ldots, 1/D)$. We may now conclude that \eqref{usc_real} converges to a closed connected invariant set associated with 
\begin{equation} \label{ana_eq2}
\dot{x}(t) \in diag(1/D, \ldots, 1/D) \nabla_x F(x(t)) + \overline{B}_\epsilon (0).
\end{equation} 

\vspace*{4pt}

\noindent
 Before we present our main result, let us state Theorem 2 from Chapter 6 of \cite{aubin2012differential}:
 \textit{
 Let $H$ be a upper semicontinuous set-valued map, from a compact $\mathcal{K} \subset \mathbb{R}^d$ to $\mathbb{R}^d$ such that $H(x)$ is compact convex for every $x \in \mathbb{R}^d$. Further, let $x(\cdotp)$ be a solution to $\dot{x}(t) \in H(x(t))$ that converges to $x^* \in \mathcal{K}$. Then, $x^*$ is an equilibrium point of $H$.
 }
 
\vspace*{4pt} 
 
 As stated earlier, the asymptotic behavior of \eqref{usc_real} and \eqref{ana_eq2} are identical. It follows from $(A3)$ that there exists $R< \infty$ such that $x^n \in \overline{B}_R(0)$ for all $n \ge 0$. Since any limit-point of the $x^n$ sequence belongs to $\overline{B}_R(0)$, a solution to \eqref{ana_eq2} tracked by \eqref{usc_real}, converges to $\overline{B}_R(0)$. The conditions of the above stated theorem are satisfied, and \eqref{ana_eq2} converges to an equilibrium point of $diag(1/D, \ldots, 1/D) \nabla_x F + \overline{B}_\epsilon (0)$. In particular, \eqref{usc_real} converges to $x^\infty$ such that $0 \in diag(1/D, \ldots, 1/D) \nabla_x F(x^\infty) + \overline{B}_\epsilon (0)$. Hence \eqref{usc_real} converges to a small neighborhood of the minimum values of $F$. Further, this neighborhood is a function of $\epsilon$, the asymptotic norm-bound on the additive error terms. The following theorem is an immediate consequence of the hitherto presented discussions. Again, the reader is referred to Section 4 of \cite{ramaswamy2018asynchronous} for details.
 \begin{theorem} \label{ana_main}
 Under $(A1)-(A5)$, the general asynchronous optimization algorithm given by \eqref{usc_real} converges to a small neighborhood of the minimum set of $F$. The size of this neighborhood is a function of the asymptotic norm-bound on the additive error terms.
  \end{theorem}
  \begin{remark}
  \label{ana_remark2}
  If $\epsilon = 0$ in $(A5)$, i.e., there are no additive errors, then we get that \eqref{usc_real} converges to $x^\infty$, an equilibrium point of $diag(1/D, \ldots, 1/D) \nabla_x F$. Since, $diag(1/D, \ldots, 1/D) \nabla_x F(x^\infty) = 0$, we get that $\nabla_x F(x^\infty) = 0$. In other words, when $\epsilon = 0$, iteration \eqref{usc_real} converges to a minimum of $F$. As $\epsilon$ increases, the algorithm is expected to converge farther from the minimum set. Therefore, when $\epsilon \uparrow \infty$ a minimum of $F$ cannot be found.
  \end{remark}
\section{The Algorithm} \label{sec_algo}
In this section, we use the theory developed in Section~\ref{sec_gen_analysis} to develop practical verifiable sufficient conditions on network topology and quality, and the CP, see (A6) in Section~\ref{sec_algo_ana}. This is one of the major contributions of this paper. \emph{To the best of our knowledge, the network condition (A6) is the most practical, weak and easy to verify assumption in the literature}. Since we do not assume anything substantial with regards to network topology, we present a CP that is topology agnostic. In addition to being easy-to-implement, our CP is more widely applicable even to scenarios with partial or zero knowledge of the topological properties. Motivated by the theory presented in Section~\ref{sec_gen_analysis} and the form of iteration \eqref{usc_real}, we present Algorithm~\ref{algo1}. It codifies, succinctly, the CP and the gradient computations of various agents.

The organization of this section is as follows. First, we present Algorithm~\ref{algo1}, the distributed delayed gradient-based optimization algorithm. Next we present (A6), the network-related assumption. Finally, we analyze Algorithm~\ref{algo1} under (A6). Specifically, we show that assumptions (A1)-(A5) are satisfied, and use Theorem~\ref{ana_main} to understand its long-term behavior.

Recall that a multi-agent system and its associated communication network can be modeled as a digraph (directed graph) $\mathcal{G} \coloneqq (\mathcal{V}, \mathcal{E})$, i.e., $\mathcal{G}$ is the MAS-graph. Vertex $i \in \mathcal{V}$ \textit{iff} agent-i belongs to our system. Edge $(i, j) \in \mathcal{E}$ \textit{iff} agent-i can communicate to agent-j directly. Define $\mathcal{N}^{in}(i) := \{j \mid (j, i) \in \mathcal{E}\}$ and $\mathcal{N}^{out}(i) := \{j \mid (i, j) \in \mathcal{E}\}$, for every $i \in \mathcal{V}$. If all channels are bidirectional, then $\mathcal{N}^{in}(i) = \mathcal{N}^{out}(i)$ for every $i$. If every agent can communicate with every other agent directly, then $\mathcal{N}^{in}(i) = \mathcal{N}^{out}(i) = \mathcal{V} \setminus \{i\}$.

%Recall that agents in the given $D$-agent system improve their local estimates $x^n _i$, using estimates from other agents in the system. These estimates are communicated using, possibly unidirectional, wireless channels. We consider the general setting where every pair of agents may not be able communicate with each other directly. However, we assume that they can communicate with each other indirectly via other agents. As a consequence, any agent updates its local estimate, using possibly old information.

We assume that agent-i maintains the latest available estimates, $\hat{X}_i := (\hat{x}_{1i}, \ldots, \hat{x}_{Di})$, from other agents in the system. At time $n$, $\hat{x}_{1j} = x^{m(j)}_j$ for some $0 \le m(j) \le n$ and $1 \le j \le D$. To facilitate consistent updates for $\hat{X}_i$, $1 \le i \le D$, each agent appends its estimate with a local timestamp, before transmitting across the network. Without loss of generality, agent-i has instant access to its local estimates, i.e., $\hat{x}_{ii} = x^n_i$ at time $n$. Also without loss of generality, we assume that the agents exchange information in accordance to the network clock, although they update local estimates in accordance to local clocks. Specifically, agent-i sends $\hat{X}_i$ to every agent in $\mathcal{N}^{out}(i)$ and recieves $\{ \hat{X}_j \mid j \in \mathcal{N}^{in}(i)\}$ at every tick of the network clock. Using $\{ \hat{X}_j \mid j \in \mathcal{N}^{in}(i)\}$ and the time-stamps contained within them, agent-i updates its list $\hat{X}_i$. Note that both local clocks and network delays lead to asynchronicity between agents. Let us codify these ideas into the following algorithm.
\begin{algorithm} 
\SetAlgoLined
Initialize local decision variable estimate $\hat{x}_i$ \;
 Synchronize with other agents to initialize $\hat{X}_i$ \;
 \For{every tick of the network clock}
 {
 \If{the local clock also ticked}{
 1. Obtain network sample $\xi$ \;
 2. $\hat{x}_i \leftarrow \hat{x}_i - a(\nu(n, i)) \nabla_{x_i} f(\hat{X}_i, \xi)$ \;
 3. Update $\hat{x}_{ii}$ of $\hat{X}_i$ to the new $\hat{x}_i$ and append the local timestamp $\nu(n, i)$ \;
 }
 4. Send $\hat{X}_i$ to agent-j, for $j \in \mathcal{N}^{out}(i)$ \;
 5. Receive $\hat{X}_j$ from agent-j, for $j \in \mathcal{N}^{in}(i)$ \;
 6. Update $\hat{X}_i$ using $\{\hat{X}_j \mid \hat{X}_j \text{ successfully received}, \ j \in \mathcal{N}^{in}(i)\}$  \;
 }
 \caption{Optimization algorithm w.r.t. agent-i}
 \label{algo1}
\end{algorithm}

Since we assume imperfect communication, the estimate-list sent by agent-j may not reach agent-i. In other words, in step 5 of Algorithm 1, agent-i may only receive a strict subset of $\hat{X}_j$s from its neighbors. Using this, it updates $\hat{x}_{ki}$, the estimate associated with agent-k, to that $\hat{x}_{kj}$ which is appended with the largest $\nu(m, k)$ (local timestamp of agent-k), that is also greater than the $\nu(m, k)$ associated with the current estimate $\hat{x}_{ki}$,  where $m \le n$, $1 \le k \le D$ and $j \in \mathcal{N}^{in}(i)$. 
\begin{remark}
Algorithm 1 states that all agents send and receive information in accordance to the network clock (steps $4$ and $5$). However, this requirement can be easily relaxed to obtain identical convergence results. In fact, agent-i is only required to send information every time its ``list'' $\hat{X}_i$ is updated, or in the case of repeated failed transmissions. The list is then updated, at every tick of the local clock and when information is received from its neighbors.
\end{remark}
\subsection{Network requirement (A6) and convergence of Algorithm~\ref{algo1}} \label{sec_algo_ana}
Before proceeding, let us introduce a few useful graph theoretical terms \cite{west1996introduction}. A path $\mathcal{P}$ is a sequence of vertices such that any vertex is connected to the next vertex by an edge, i.e., $\mathcal{P} = i_0, \ldots, i_m$ and $(i_j , i_{j+1}) \in \mathcal{E}$ for $0 \le j \le m-1$. The length of a path is given by the number of edges in it, and an edge is trivially a length-1 path. We use $iPj$ to represent a path whose first vertex is $i$ and last vertex is $j$. We use $p(i\mathcal{P}j, x)$ to represent the probability of successfully transmitting information $x$ from agent-i to agent-j, using path $\mathcal{P}$, in the shortest possible time. In other words, $p(i\mathcal{P}j, x)$ represents the probability that every link along $\mathcal{P}$ transmits successfully, one after the other. \textit{If the wireless channels are modeled as Markov fading channels then $p(i\mathcal{P}j, x)$ may change over time.}

To apply the analysis from Section~\ref{sec_ana} to understand Algorithm~1, we need the following simple assumption on the underlying transmissions. 
%In the previous section, we presented a decentralized consensus algorithm which can account for unbounded stochastic delays and network fluctuations. The agents try to achieve consensus by performing gradient descent, based on information obtained from their neighbours. In this section, we show that this algorithm indeed converges under the following mild assumption:

\vspace*{5pt}
\begin{itemize}
\item[\textbf{(A6)}] \begin{itemize}
\item[(i)] For every $1 \le i, j \le D$, there exists a path connecting agent-i to agent-j, $i\mathcal{P}^{ij} j$, such that $p(i\mathcal{P}^{ij}j, x) > 0$.
\item[(ii)] At any time, the success or failure of transmissions along different edges are independent.
\end{itemize}
\end{itemize}
\vspace*{5pt}

Informally (A6) requires that: (a) there is a positive probability of transmitting information between any two agents and (b) the success / failure of transmissions over one edge does not interfere with transmissions over the others. Clearly (A6)(ii) is the harder requirement, as compared to (A6)(i), to ensure. One way, is by letting the channel frequencies be sufficiently separated, so as to cause zero/minimal interference with each other. In typical real-world applications such as industries, the network topology is fixed and the channel quality varies with the system state. In particular, at any time-step, transmissions along different channels are independent when conditioned on the current system state, see \cite{quevedo2012state} for details. The independence requirement in (A6)(ii) can be readily relaxed to the aforementioned conditional independence for the industrial application.

To apply the analysis from Section~\ref{sec_ana} to understand Algorithm 1, we must verify (A1)--(A5). Since there are no additive errors, (A5) is trivially satisfied. (A1) is satisfied by design. If the experimenter enforces that the agents must synchronize their clocks with the network clock, then (A4) is satisfied. While this enforcement is standard, there are several other ways to ensure (A4). Note that (A4) merely requires that all agents are updated in the same order of magnitude. This condition is stated by requiring that the number of agent-i updates is in $\Theta(n)$, $1 \le i \le D$. Note that the $\Theta$ notation is used to indicate that the two quantities involved have similar asymptotic growth.
Since, the experimenter is typically at liberty to choose the step-size sequence, (A2)(i)-(iv) are trivially satisfied. We refer the reader to sufficient conditions in \cite{ramaswamy2018asynchronous} to ensure $(A3)$. \emph{It is now left to show that (A2)(v) is satisfied}.
\begin{lemma}
\label{algo_lemma}
Under (A6), Algorithm~\ref{algo1} satisfies (A2)(v).
\end{lemma}
\begin{proof}
We prove the statement of the lemma by finding a random variable $\overline{\tau}$ that stochastically dominates $\tau_{ij}(n)$ such that $\mathbb{E} \overline{\tau} ^2 < \infty$, $1 \le i, j \le D$ and $n \ge 0$. Then, (A2)(v) is satisfied for any $\eta \in (\nicefrac{1}{2}, 1)$. Fix $1 \le i,j \le D$ and $n \ge 0$. Recall that $\tau_{ij}(n)$ denotes the age of the agent-i-estimate available at agent-j at time $n$, i.e., the age of $\hat{x}_{ij} \in \hat{X}_j$. Let $\mathcal{P}$ be the shortest length path satisfying (A6)(i), and let $l$ be its length. We define a random variable $\hat{\tau}_{ij}(n)$, to denote the age of the latest agent-i-estimate that was successfully transmitted using $\mathcal{P}$ in the shortest possible time, i.e., in $l$ steps. Since $\tau_{ij}(n)$ denotes the age of the last agent-i-estimate that was successfully transmitted, we have that $\tau_{ij}(n) \le \hat{\tau}_{ij}(n)$ a.s. Hence, it is enough to find a random variable $\overline{\tau}$ that stochastically dominates $\hat{\tau}_{ij}(n)$.

Let $A^k _{ij}$ denote the event that the $\hat{X}_i$-vector sent at time $k$ (of the network clock) was successfully transmitted via $\mathcal{P}$ in $l$ steps. Then,
\begin{equation}
\label{algo_eq1}
P(\hat{\tau}_{ij}(n) > m) = P \left( \substack{n-l \\ \bigcap \\ t=n-m} \left( A_{ij}^t \right)^c \right), 
\end{equation}
where $\left( A_{ij}^t \right)^c$ denotes the complement of event $A^t _{ij}$. We claim that the events $\{ A^t _{ij} \mid n-m \le t \le n-l \}$ are independent. To see this, we begin by observing that $A^t _{ij}$ unfolds over time starting at $t$ and ending at most before $t+l$. Further, at any given time instance each of these events can be associated with different edges. The independence of the above mentioned set of events now follows from the independence of the associated edges, i.e., (A6)(ii). Now, equation \eqref{algo_eq1} becomes:
\begin{equation*}
P(\hat{\tau}_{ij}(n) > m) = \substack{n-l \\ \pro \\ t=n-m} P \left( \left( A_{ij}^t \right)^c \right). 
\end{equation*}
It follows from (A6)(i) that there is $p$, such that $ \sup \limits_{n-m \le t \le n-l}P \left( \left( A_{ij}^t \right)^c \right) \le p < 1$. Hence we get,
\begin{equation} \label{algo_eq2}
P(\hat{\tau}_{ij}(n) > m) \le  p^{m - l + 1}.
\end{equation}
Please note that the above inequality is independent of $n$. Further, without loss of generality, the above inequality holds for every $(i,j)$ pair with the same $p$ and $l$. 

Let us define $\overline{\tau}$ by describing its cumulative distribution function as follows: $P(\overline{\tau} \ge 0) := 1$ and $P(\overline{\tau} \ge m) := p^{m - l}$, $m \ge 1$. It now follows from \eqref{algo_eq2} that $\overline{\tau}$ stochastically dominates $\tau_{ij}(n)$ for all $1 \le i,j \le D$ and $n \ge 0$. 

It is left to show that $\mathbb{E} \overline{\tau} ^2 < \infty$. To do this, we observe that:
\[
\mathbb{E} \overline{\tau} = \sum \limits_{m = 0}^\infty m P (\overline{\tau} \ge m) \le \sum \limits_{m = 0}^\infty m p^{m-l}.
\]
The required result now follows from the simple observation that:
\[
\sum \limits_{m = 0}^\infty m p^{m-l} \le \int \limits_{0}^\infty x \left(\nicefrac{1}{p} \right)^{ -x + l} \ dx < \infty.
\]
\end{proof}
Hence, we have shown that Algorithm 1 satisfies the assumptions required to apply Theorem~\ref{ana_main}. The following is therefore an immediate consequence of the aforementioned theorem and Remark~\ref{ana_remark2}.
\begin{theorem}
\label{algo_main}
Algorithm 1 finds a minimum of $F$ when (A1), (A2)(i-iv), (A3) and (A6) are ensured.
\end{theorem}
%\subsection{Superfluity of (A6)(ii)}
\section{Empirical Results}
\label{sec_empirical}

We now present experimental results to substantiate the hitherto developed theory. We consider a 16-agent system such that $\mathbb{R}^2$ is the local decision space of every agent in the system. The agents are connected by a wireless network whose topology is time-varying. This is simulated by randomly choosing, at every step, between $N=4$ digraphs $\{\mathcal{G}_n \equiv (\mathcal{V}, \mathcal{E}_n)\}_{n=1}^4$. Since we assume full asynchronicity, it is important to note that the network topology is randomly switched at every tick of the network clock. For $1 \le i \le 16$, define $\rho_i \coloneqq \frac{a_i}{b_i}$, where $a_i$ is the number of agent-i clock ticks for every $b_i$ network clock ticks. We call $\rho_i$ the $i^{th}$ tick-ratio. For our experiments $\{\rho_i\}_{i=1}^{16}$ are chosen to be independent random variables, each sampled uniformly from $[1/5, 1]$. Note that $\{\rho_i\}_{i=1}^{16}$ is used to simulate the asynchronicity between various agents in the system, and with the wireless network. We choose them in $[1/5, 1]$, since we want to satisfy (A4).

For added complexity, we constrain each of the $4$ network topologies to contain exactly $8$ edges, i.e., $\left| \mathcal{E}_n \right| = 8, \ 1\le n \le 4$. Each edge represents a communication link that is modeled as a Markov fading channel with state space $\mathcal{Z} = \left\{z_0,z_1, \ldots, z_{K-1} \right\}$, parameterized by tuples 
$(T,p,e)$, where $T$ denotes the transition probability matrix, $p$ denotes the steady state probability vector $p$ (defined by the equilibrium condition $p^\top T = p^\top$), and $e$ denotes the crossover probability vector, see \cite{wang1995finitestate} for details. In each channel state $z_d$, fading results in communication drop outs with probability in accordance to the $d$-th component of the crossover probability vector $e$.
%For each $\mathcal{E}_k$, we consider a fixed assignment of channels to edges over the simulation horizon.

% cost function
Let us say that all agents try to simultaneously minimize the perturbed ``average network distance''. The average is taken over a network switching process involving $4$ different topologies. Then, the global objective function is given by:
\[
F(x) \coloneqq \mathbb{E}_{\xi} f(x , \xi), \text{ where}
\]
$f(x, \xi) \coloneqq x^T \left[L_\xi + 0.1 \mathbb{I} \right] x$; $x = (x_1, \ldots x_{16})$ is the appended decision vector; $L_{\xi}$ is the Laplacian matrix associated with the graph $\mathcal{G}_{\xi} \equiv (\mathcal{V}, \mathcal{E}_{\xi})$; $\xi \in \{1, \ldots, 4\}$ is the random variable associated with network switching; and $\mathbb{I}$ is a square matrix of order $32$. The Laplacian $L_\xi$ associated with a network topology is positive semi-definite. When perturbed by $0.1 \mathbb{I}$, it becomes positive definite. We work with this perturbed ``average network distance'' measure to force the associated optimization problem to have a unique global minimum. Working with a perturbed Laplacian is common in literature, see \cite{rocha2016fiedler, bapat2001perturbed}. There are several ways to perturb a Laplacian, we do the above as it is simple, and is sufficient to illustrate the properties of Algorithm~\ref{algo1}.

The analyses from Sections~\ref{sec_ana} and \ref{sec_algo} are applicable, provided assumption (A6) is satisfied for the experimental setup described above. Convergence to the minimum of $F$ is then a consequence of Theorem~\ref{algo_main}. To show that (A6) is satisfied, we consider the union graph $\hat{\mathcal{G}} \coloneqq \left(\mathcal{V}, \bigcup \limits_{n=1}^{4} \mathcal{E}_n  \right)$. At the tick of the network clock, the sample of $\xi$ determines which among the $4$ available network topologies is to be activated. For $\hat{\mathcal{G}}$, this means that a subset of edges $\mathcal{E}_\xi$, corresponding to topology $\mathcal{G}_\xi$, is ``active''. This means that all possible transmissions, at that time, are along these active edges. There is no information transmission along the remaining ``inactive'' edges. We however, take the quasi-equivalent view that transmissions along inactive edges are also possible, only to fail with probability $1$. This view allows us to move from time-varying network topologies to considering a static network topology, given by $\hat{\mathcal{G}}$, that does not change over time.
%At any given time, depending on the network topology chosen, one may say that a subset of edges corresponding to that network are ``active'' in $\hat{\mathcal{G}}$. Further, transmissions along all other ``inactive'' edges may be considered as having failed. In this sense, $\hat{\mathcal{G}}$ is the required static network topology. 
(A6)(i) is satisfied when $\hat{\mathcal{G}}$ is strongly connected. Recall that a digraph is strongly connected when any vertex is reachable from any other vertex. (A6)(ii) is ensured if we assume that all channels are uncorrelated. In practice, this assumption is satisfied when the channel frequencies are sufficiently separated.

Fig.~\ref{fig: network_topologies} illustrates the union-graph of the four network topologies used in our experiments. Algorithm~1 is run for $5000$ steps, using the following step-size sequence:
\begin{equation*}
a(\nu(n,i)) = \frac{1}{\frac{\nu(n,i)}{50} + 50},\ n \le 5000 \text{ and }1 \le i \le 16
\end{equation*}
The initial estimates/ decision variables of all $16$ agents are randomly chosen from the annulus centered at the origin with inner ring radius of $500$ and outer ring radius of $1000$. The Laplacian matrices were perturbed such that the origin is the unique minimizer of $F$. The results of our experiments are illustrated in Figures~\ref{fig: plane} and \ref{fig: cost_conv}. The various agent estimate-trajectories, in $\mathbb{R}^2$, are illustrated in Fig.~\ref{fig: plane}. The estimates converge to the local optima, which put together minimize the ``average network distance''. This is illustrated by the convergence of the perturbed ``average network distance'' towards the origin, in Fig.~\ref{fig: cost_conv}. 
% 
 %$\left( \mathcal{V}, \cup_{i=1,\ldots,N} \mathcal{E}_i \right) $ needs to be strongly connected, which implies that (A6)(i) is satisfied.

\begin{figure}[!t]
\centering
\includegraphics[width=3in]{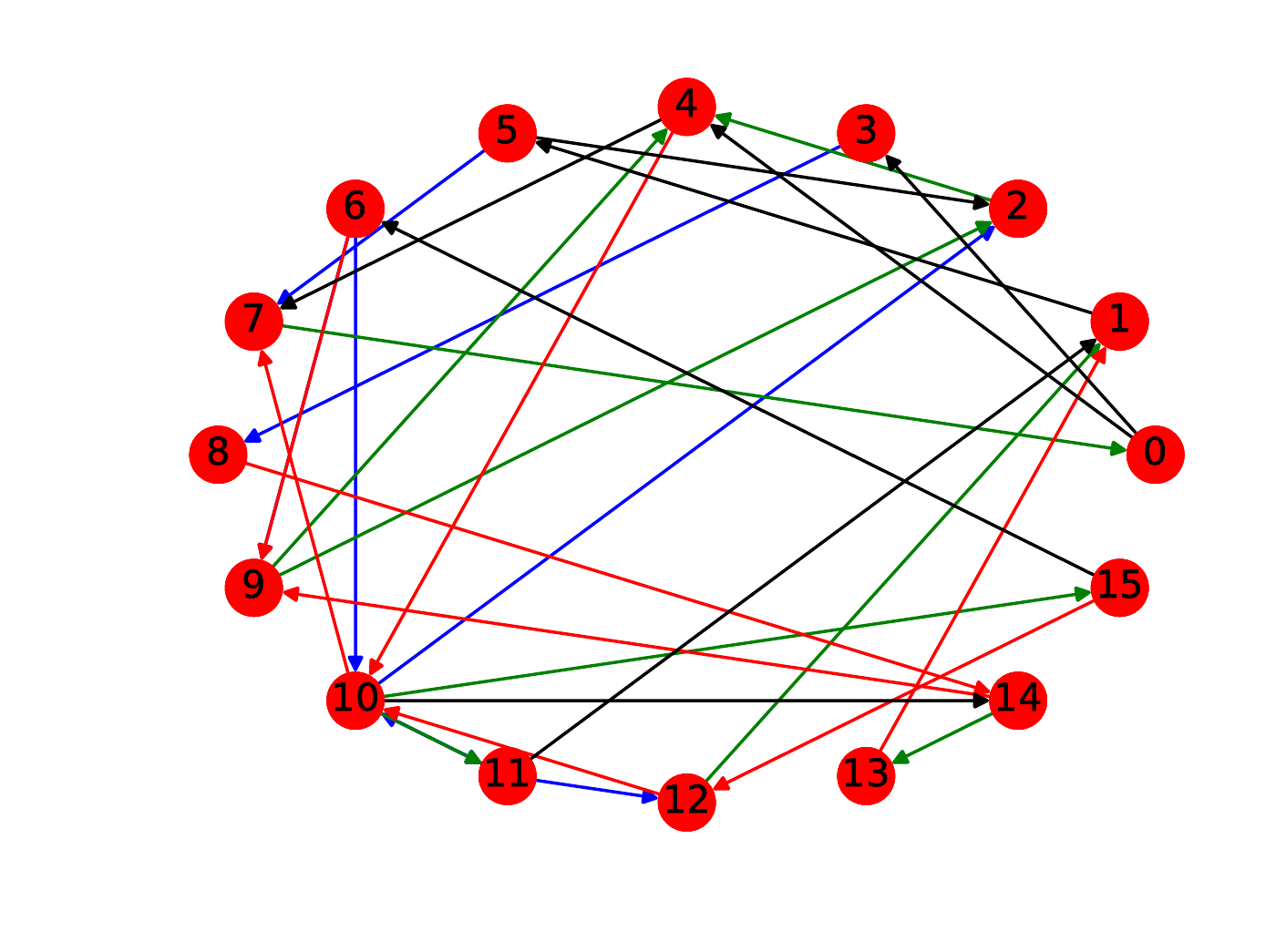}
\caption{Four super-imposed network topologies whose union is a strongly connected graph.}
\label{fig: network_topologies}
\end{figure}

\begin{figure}[!t]
	\centering
	\includegraphics[width=3in]{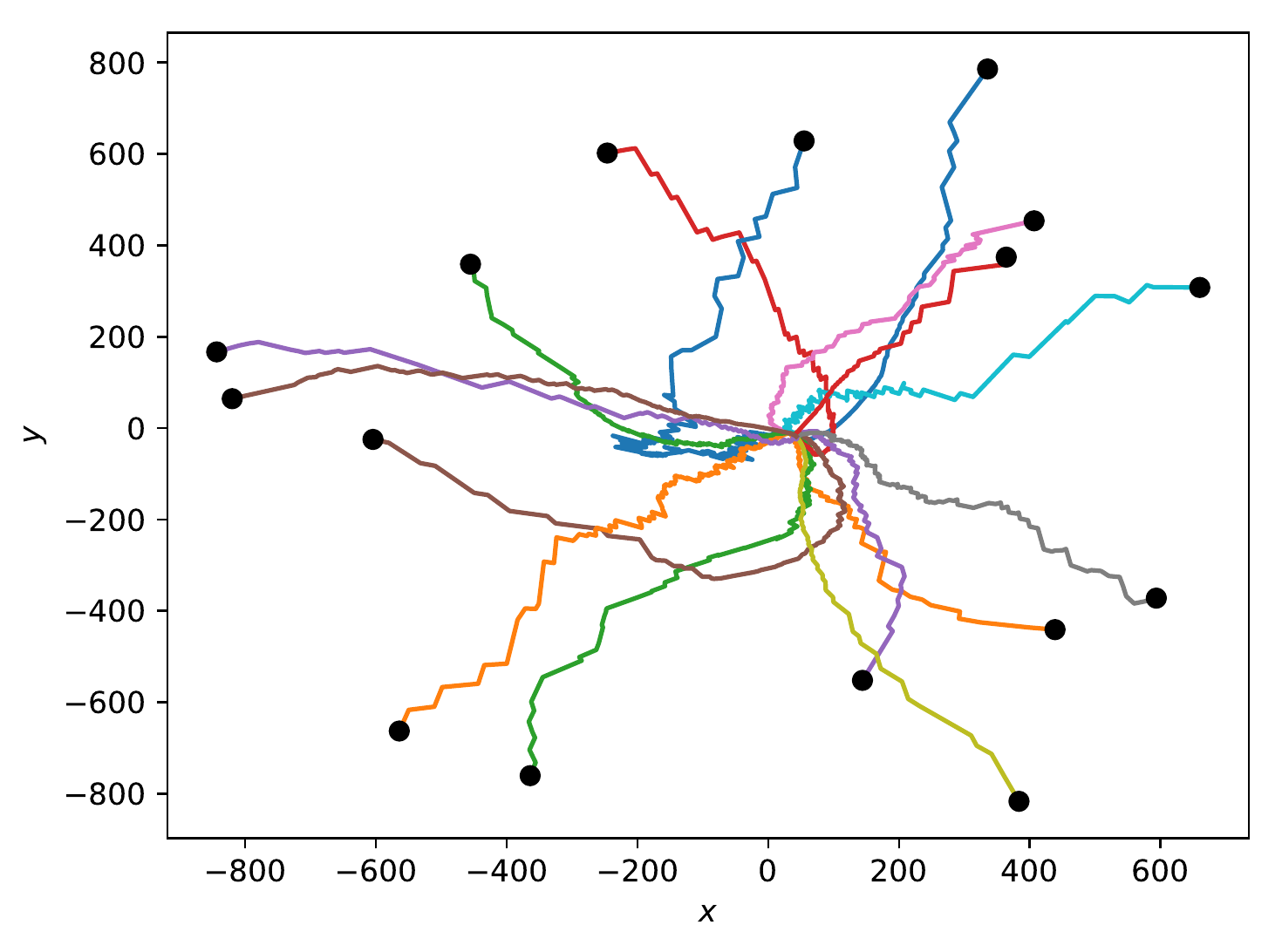}
	\caption{Sixteen agent estimate-trajectories are illustrated in $\mathbb{R}^2$ using different colors.}
	\label{fig: plane}
\end{figure}

\begin{figure}[!t]
	\centering
	\includegraphics[width=3in]{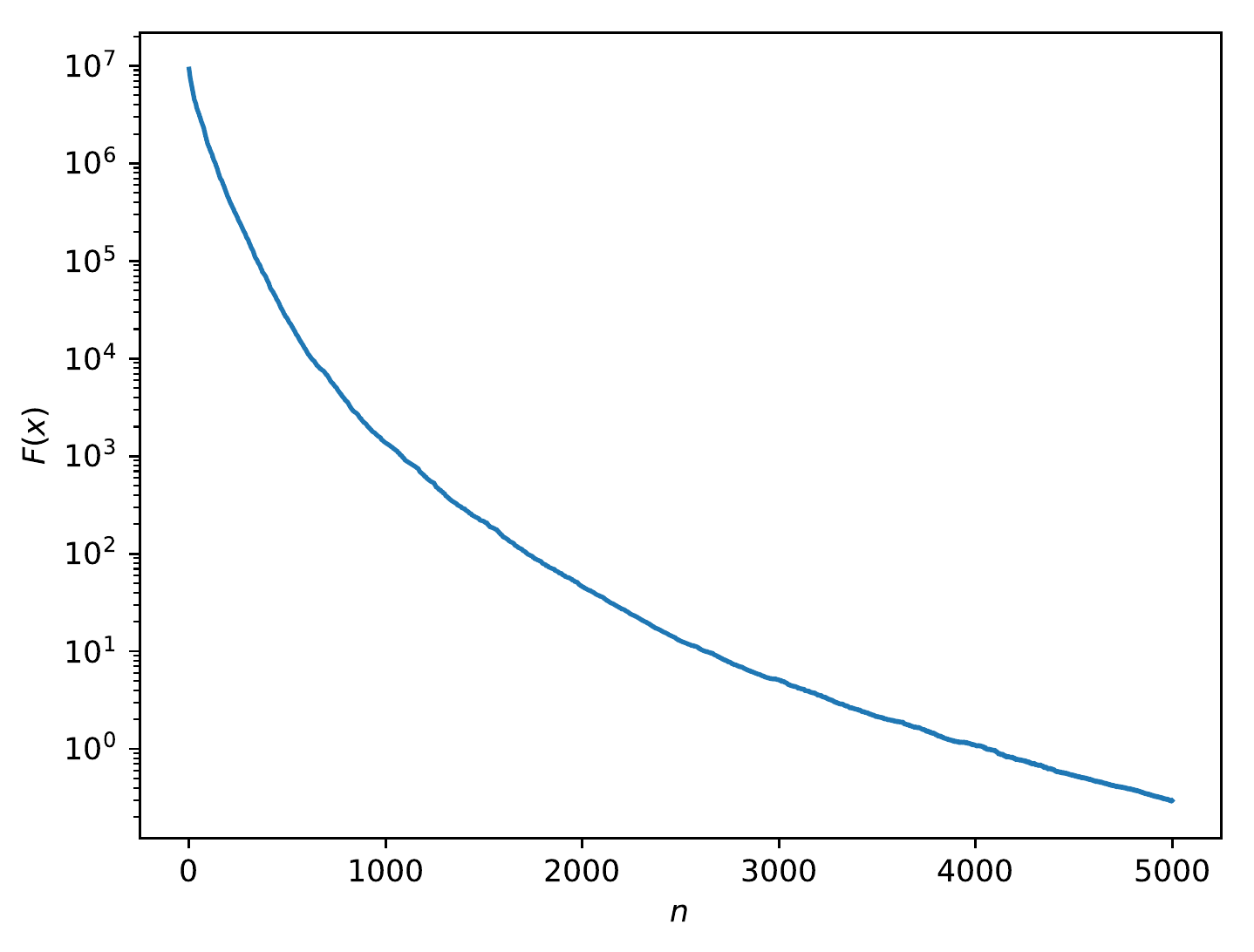}
	\caption{Convergence of the perturbed ``average network distance'' trajectory (F).}
	\label{fig: cost_conv}
\end{figure}

\section{Special Case: Cumulative Consensus}
\label{sec_cons}
Consensus is an important optimization problem that commonly arises in multi-agent systems. Popular forms of consensus include cumulative and average consensus. In this section, we focus on the cumulative consensus problem. Here, the D-agent system cooperatively tries to solve the following global optimization problem:
   \begin{equation} \label{cons_eq}
   \underset{x \in \mathbb{R}^d}{\text{argmin }} \sum \limits_{i=1}^D f_i(x),
   \end{equation}
   where $f_i: \mathbb{R}^d \to \mathbb{R}$ is a local function that is only accessible to agent-i; $x \equiv (x(1), \ldots, x(D))$ and $x(i)$ is interpreted as a local control variable of agent-i. The global objective function $F$ is such that $F: x \mapsto \sum \limits_{i=1}^D f_i(x)$.
   
   We begin by noticing the key differences between problems \eqref{intro_opt_prob} and \eqref{cons_eq}. In \eqref{intro_opt_prob}, noisy samples of the global objective function is available to all agents. In \eqref{cons_eq}, each agent only observes its local objective, and in this sense the global objective is partially observed. On the surface, the consensus problem may seem harder than \eqref{intro_opt_prob} due to partial observability. However, we present simple modifications to Algorithm~\ref{algo1} to solve the consensus problem. Further, we show that the analyses from Section~\ref{sec_ana} and \ref{sec_algo} describe completely the behavior of this new modified algorithm.

   The solution to the cumulative consensus presented in this section involves choosing one of the agents as an ``arbiter'', who is then responsible for all computations involved in solving the given consensus problem. All other agents in the system support the arbiter by supplying it with information required to perform the aforementioned computations. A completely distributed approach, in the vein of \eqref{usc_real}, where the computational load is shared by all agents, will turn out to be a simple extension of the arbiter-approach. We present the former for the sake of simplicity and clarity. It may be noted that the arbiter-approach is susceptible to the ``single point of failure'' problem. However, it should also be noted that the arbiter-approach uses much lesser resources than the fully distributed approach. This is because all agents need only to communicate with the arbiter, not with each other. 
   \begin{remark}
   Suppose that all paths connecting any two vertices in the associated communication graph pass through a single vertex/ agent. Then, that agent is clearly a single point of failure even in a distributed approach. For this case, the arbiter approach is in fact optimal when the said agent is chosen as the arbiter. 
   \end{remark}
   
   Motivated by (A1), we make the following assumption regarding the local objectives: \\
   \vspace*{.05cm}\\
   \textbf{(C)} For $1 \le i \le D$, $\nabla f_i$ is locally Lipschitz continuous.\\
   \vspace*{.05cm}\\
 Similar to Algorithm~1, we associate a digraph $\mathcal{G}$ with the multi-agent system at hand. We assume that $\mathcal{G}$ satisfies (A6) and let agent-1 be the arbiter. Clearly agent-1 is strongly connected to every other agent. In the centralized paradigm, the arbiter is assumed to have complete knowledge of all local objectives $\{f_i\}_{i=1}^D$, and it may solve the consensus problem as follows:
 \begin{equation}\label{eq_central_arbit}
 x^{n+1} = x^n - a(n) \sum \limits_{i=1}^D \nabla f_i (x^n).
 \end{equation}
 However, as we are in a decentralized setting, agent-1 only has knowledge of $\nabla_x f_1$. It must obtain all additional gradients from other agents in the MAS, if it is to execute an iteration similar to \eqref{eq_central_arbit}. We have assumed that the network is unreliable, and that every pair of agents may not be connected by a dedicated wireless channel. Hence, gradient information obtained by agent-1, from agent-j, $j \neq 1$, may be old. At time $n$, instead of \eqref{eq_central_arbit}, agent-1 (arbiter) executes:
 \[
 x^{n+1} = x^n - a(n) \sum \limits_{i=1}^D g_1^j (n),
 \]
 where $g_1^j (n)$ is the, possibly delayed, agent-j gradient information available with the arbiter at time $n$. 
 
 Similar to Algorithm~1, agent-i updates and \emph{forwards} the estimate-vector $\hat{X}_i \coloneqq \left(\hat{x}_i, g_i^2, \ldots, g_i^D \right)$; where $\hat{x}_i$ is the latest estimate of the solution to the consensus problem available with agent-i, and $g_i^j$ is the latest agent-j gradient information that is available with agent-i. Specifically, at time $n$, $g_i^j = \nabla f_j (x^{n - \tau(i,j)})$, for some delay $0 \le \tau(i,j) \le n$, for $2 \le j \le D$. At every tick of the network clock, agent-i sends its estimate vector $\hat{X}_i$ to all agents belonging to $\mathcal{N}^{out}(i)$. At the same time, it receives estimate-vectors $\left\{ \left( \hat{x}_j, g_j^2, \ldots, g_j^D \right) \mid j \in \mathcal{N}^{in}(i) \right\}$ from its ``in-neighbors''. Due to the unreliable nature of the wireless network, it  may only receive information from a strict subset of in-neighbors. This is then used to update $\hat{X}_i$ at the next tick of the agent-i clock, similar to step-6 in Algorithm~\ref{algo1}. Finally, agent-i calculates $\nabla f_i (\hat{x}_i)$ and updates $g_i^i$ to this new gradient value. In the case of the arbiter, $g_i^j$, for $2 \le j \le D$ are updated and the following descent step is executed:
 \begin{equation}
 \label{agent1_it}
 x^{n+1} = x^n -  a(\nu(n, 1)) I(1 \in Y_n) \sum \limits_{i =1}^D g^j_1. 
 \end{equation}
 Then, the arbiter revises $\hat{x}_1$ to $x^{n+1}$ and $g_1^1$ to $\nabla f_1(x^{n+1})$. Although the agents forward information in accordance to the network clock, local computations and updates are governed by local clocks. This naturally leads to asynchronicity between agents, that is captured by the $Y_n$ process in \eqref{agent1_it}. Please note that we have considered asynchronicity for the sake of generality and completeness. To understand the main ideas of this section, the reader may simply omit the $Y_n$ process, and assume that the agents are synchronized their clocks.

  If we consider the gradient information used in \eqref{agent1_it}, then $g_1 ^i = \nabla f_j \left(x^{n - \tau_j} \right)$, for some delay $0 \le \tau_j \le n$, $1 \le j \le D$. Each delay random variable $\tau_j$ can be decomposed into $\tau_j^1$ and $\tau_j^2$, i.e., $\tau_j = \tau_j^1 + \tau_j^2$. Here $\tau_j^1$ is the number of steps taken by $x^{n - \tau_j}$ to reach agent-j from agent-1, and $\tau_j^2$ is the time taken for $\nabla f_j \left( x^{n - \tau_j} \right)$ to return to agent-1. Since the events associated with $\tau_j^1$ and $\tau_j^2$ are disjoint in time, it follows from $(A6)(ii)$ that $\tau_j^1$ and $\tau_j^2$ are independent. Using arguments similar to the proof of Lemma~\ref{algo_lemma}, we can show that there is a random variable $\overline{\tau}$ that stochastically dominates $\tau_j^1$ and $\tau_j^2$, for $1 \le j \le D$, such that $\mathbb{E} \overline{\tau}^2 < \infty$. Hence $2 \overline{\tau}$ stochastically dominates $\tau_j$ and $\mathbb{E} \left[2\overline{\tau} \right]^2 < \infty$, and \eqref{agent1_it} satisfies $(A2)(v)$. Now, we can readily apply analyses from Sections~\ref{sec_ana} and \ref{sec_algo} to understand its long-term behavior, and the following theorem is immediate.
  \begin{theorem}
  \label{cons_thm}
  Under (C), (A2)(i - iv), (A3), (A4) and (A6) the above described consensus algorithm, \eqref{agent1_it}, achieves consensus.
  \end{theorem}
  For easy reference, below we present the consensus algorithm, obtained by modifying Algorithm~\ref{algo1}.
  \begin{algorithm}
\SetAlgoLined
 Synchronize with other agents and initialize $\hat{X}_i$ \;
 \For{every tick of the network clock}
 {
 1. Send $\hat{X}_i$ to agent-j, for $j \in \mathcal{N}^{out}(i)$ \;
 2. Receive $\hat{X}_j$ from agent-j, for $j \in \mathcal{N}^{in}(i)$ \;
 \If{the local clock also ticked}{
 3. (i) Agent-i ($i \neq 1$) updates $\hat{x}_i$ and $g_i ^j$, $2 \le j \le D$, using $\{ \hat{X}_j \mid j \in  \mathcal{N}^{in}(i) \}$ (ii) $g_i^i \leftarrow \nabla f_i(\hat{x}_i)$ \;
 4.  Agent-1 updates $g_1 ^j$, $2 \le j \le D$, using $\{ \hat{X}_j \mid j \in  \mathcal{N}^{in}(1) \}$ (ii) $\hat{x}_1 \leftarrow \hat{x}_1 - a(\nu(n, 1)) \sum \limits_{j =1}^D g_1 ^j$ (iii) $g_1^1 \leftarrow \nabla f_1(\hat{x}_1)$ \;
 }
 }
 \caption{Consensus Algorithm}
 \label{algo2}
\end{algorithm}

\begin{remark}
 Choosing the right agent as the arbiter is problem dependent. It is best to choose one based on ``accessibility'', i.e., choose an agent that is connected by multiple short paths composed of reliable channels, to all other agents. On the other hand, an agent that is connected to most agents by a single unreliable path is a bad candidate. In this section, we have chosen agent-1 as the arbiter as we don't assume anything beyond strong connectivity of the underlying graph.
 \end{remark}
 \subsection{Empirical Results: Algorithm~\ref{algo2}}
 Although Algorithm \ref{algo2} is expected to have similar empirical behavior to Algorithm~\ref{algo1}, we present quick results of an experiment, to corroborate the discussion from Section~\ref{sec_cons}.
In our experiment, we considered a 16-agent system such that the common control variable belongs to $\mathbb{R}^{32}$. The underlying communication network used, is identical to the one from Section~\ref{sec_empirical}. Also, the asynchronicity between agents is simulated using the procedure explained in Section~\ref{sec_empirical}.

The local objective of agent-i is given by:
\begin{equation}
	f_i(x) = x^T A_i x + b_i^T x + c_i,
\end{equation}
where $A_i \in \mathbb{R}^{32\times 32}$ is a symmetric positive definite matrix, $b_i \in \mathbb{R}^{32}$ and $c_i \in \mathbb{R}$. The matrices $A_i$, the vectors $b_i$ and the scalars $c_i$ are all randomly generated. 

We therefore have the following cumulative consensus problem: $$\underset{x\in \mathbb{R}^{32}}{\text{argmin}} \sum_{i=1}^{16} f_i(x).$$ Its analytic solution is given by:
\begin{equation}
	x = - \frac{1}{2}\left( \sum_{i=1}^{16}A_i \right)^{-1} \sum_{i=1}^{16} b_i.
\end{equation}

The arbiter (Agent 1) uses the following step-size sequence to update the estimates (step-4, Algorithm~\ref{algo2}).
\begin{equation}
a(\nu(n,1)) = \frac{1}{\frac{\nu(n,1)}{5} + 150}, \qquad n \le 2500
\end{equation}
The initial estimate is randomly chosen from a ball of radius $10$ centered at the origin, and the algorithm is run for $2500$ iterates. Fig.~\ref{fig: consensus_conv} illustrates the performance of Algorithm~\ref{algo2}, when applied to solve the above described consensus problem. In the figure: (a) the dotted line represents the analytical solution to the cumulative consensus problem, (b) the dark blue line illustrates the convergence of the cumulative objective function to the analytical solution, and (c) all the other colored lines illustrate the convergence of local objective functions, of all agents in the MAS.

\begin{figure}
	\includegraphics[width=3.5in]{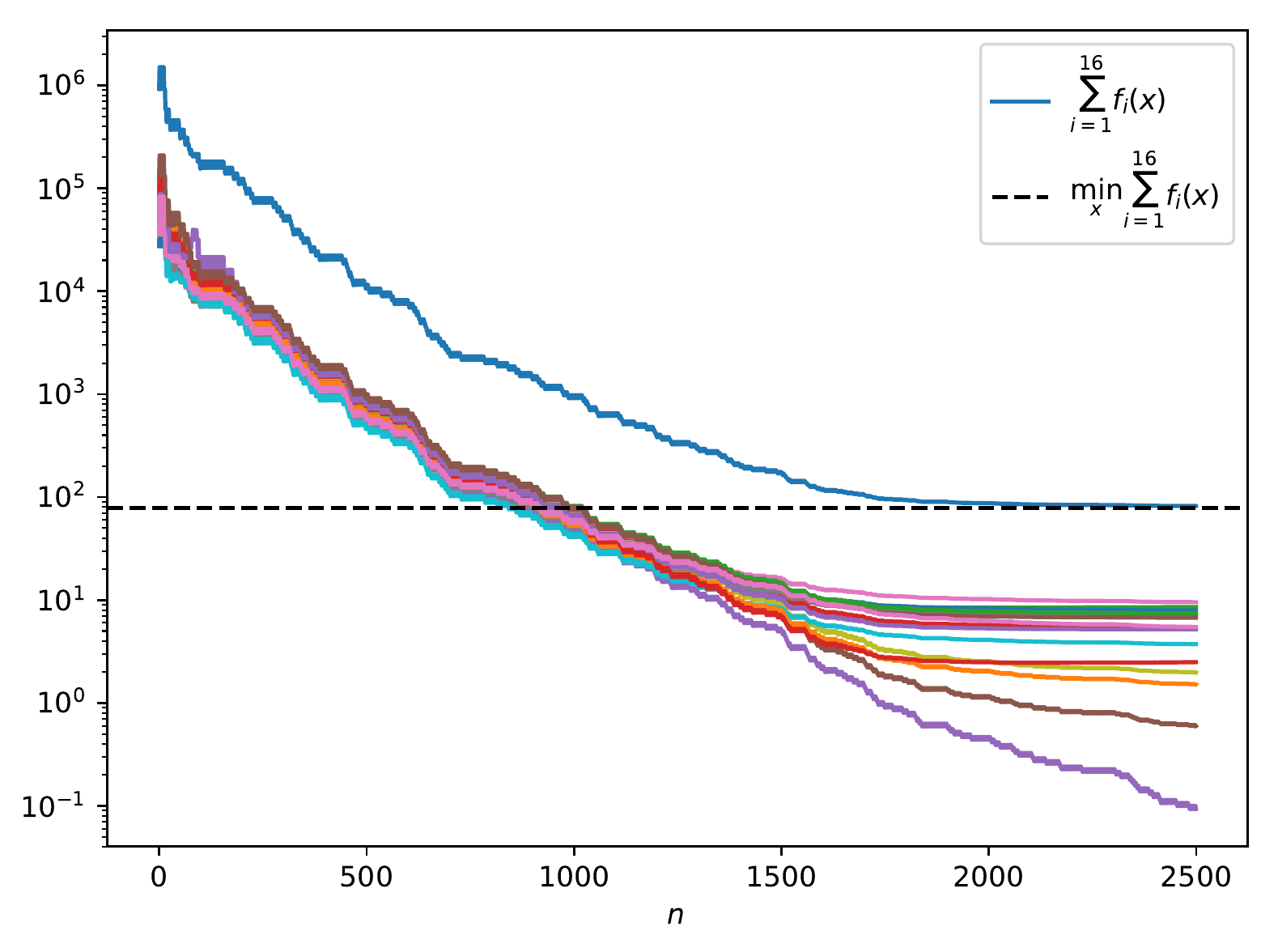}
	\caption{The dotted line represents the analytical solution to the consensus problem and the dark-blue line illustrates the convergence of the cumulative objective function, over $2500$ iterates, to the analytical minimum. Convergence of local objective functions is also illustrated using different colors.}
	\label{fig: consensus_conv}
\end{figure}

\section{Conclusions}
In this paper, we considered the problem of distributed optimization in large-scale multi-agent systems, connected by an unreliable wireless network. We presented an analytical framework to understand the long-term behavior of distributed delayed (approximate) gradient-based optimization algorithms. The theory developed, allows for a larger class of objective functions, as compared to similar results from the literature. Specifically, the presented analyses requires that the objective functions have locally Lipschitz continuous gradients. In addition to being easy to verify, this condition encompasses highly non-linear objectives. The framework also accommodates the use of approximate gradients instead of exact ones. A major contribution of this paper is in the development of practical verifiable requirements on network quality and topological properties. In particular, we merely require that the probability of successful information exchange between agents is positive. These probabilities may vary over time and there may be delays involved in the information exchange process. This is due to the unreliable nature of the underlying wireless network, in addition to possessing time-varying topological properties. In particular, when the topological switching is governed by some stochastic process, our theory holds water even when the switching process is non-stationary. To the best of our knowledge, ours is the first analyses under such weak conditions. The theory was also used to develop a simple easy-to-implement optimization algorithm that uses distributed delayed gradient information.

%In this paper, we considered the problem of distributed optimization in large-scale multi-agent systems, connected by a wireless network that is prone to delays and errors. We presented an analytical framework, based on stochastic approximation algorithms, to analyze distributed (approximate) gradient-based optimization algorithms. We also presented a simple distributed gradient based algorithm, that is convergent even in the presence of unbounded stochastic delays, with respect to information exchange between agents. The theory was used to develop weak sufficient conditions on communication, required for the convergence of the aforementioned algorithm. Specifically, we only required that the wireless network facilitate successful information exchange between agents with some positive probability. Further, there may be delays involved and the probability itself may vary over time. This accounts for a vast majority of wireless networks used in practice. For example, the network topology and quality may be time-varying, the topology switching process may have multiple stationary processes. Another assumption we made was on the independence of transmissions along different channels, and discussed its verifiability.

Finally, the theory was also used to develop an algorithm to solve the consensus problem from control. To the best of our knowledge, ours is the first algorithm, and associated analysis, that considers truly unbounded stochastic delays for information exchange between agents, in addition to time-varying topologies.

In the future, we would like to conduct a thorough study of the rate of convergence of such algorithms. In particular, we would like to study the impact of both network topology and quality on the rate of convergence. Information exchange is an imperative aspect of the algorithms presented. However, it may be expensive. We also plan to study the optimal amount of information that needs to be transmitted, for solving the optimization problem. We would explore the possibility of developing a dynamic adaptive policy for information exchange.

\bibliographystyle{plain}
\bibliography{references}

\end{document}